\documentclass[preprint,review,12pt,times,sort&compress]{elsarticle}    

\usepackage[margin=0.75in]{geometry} 
\usepackage{amsmath,amssymb,amsthm}
\usepackage{graphicx}
\usepackage[dvips]{xcolor}
\usepackage[nottoc]{tocbibind} 
\usepackage[]{hyperref} 
\hypersetup{
    bookmarks=false,         
    bookmarksopen=true,         
    bookmarksnumbered=true,     
    bookmarksopenlevel={1},       
    bookmarksdepth={2},         
    unicode=true,          
    pdftoolbar=true,        
    pdfmenubar=true,        
    pdffitwindow=false,     
    pdfstartview={FitH},    
    pdftitle={Asymptotics of determinants of Hankel matrices via non-linear difference equations},    
    pdfauthor={Estelle L. Basor, Yang Chen and Nazmus S. Haq},     
    pdfdisplaydoctitle=true, 
    pdfsubject={Mathematics},   
    pdfcreator={Nazmus S. Haq},   
    pdfproducer={BaKoma Tex}, 
    pdfkeywords={Hankel determinants} {Asymptotic analysis} {Orthogonal polynomials} {Random Matrix Theory} {Painlev� Equations}, 
    colorlinks=true,       
    filecolor=magenta,      
    urlcolor=red,           
}
\newtheorem{theorem}{Theorem}[section]

\newtheorem{lemma}[theorem]{Lemma}

\theoremstyle{remark}
\newtheorem*{remark}{Remark}

\makeatletter
\def\@eqnnum{{\normalsize \normalcolor\rm (\theequation)}}
\makeatother
\newcommand{\al}{\alpha}

\newcommand{\be}{\begin{equation}}
\newcommand{\ee}{\end{equation}}

\newcommand{\la}{\lambda}
\newcommand{\bt}{\beta}

\newcommand{\ga}{\gamma}

\newcommand{\sig}{\sigma}
\newcommand{\bea}{\begin{eqnarray}}
\newcommand{\eea}{\end{eqnarray}}
\renewcommand{\P}{\textsf{p}}

\numberwithin{equation}{section}
\usepackage{color}
\journal{Journal of Approximation Theory, \href{http://dx.doi.org/10.1016/j.jat.2015.05.002}{{198} ({2015}) {63--110}}.}

\begin{document}

\begin{frontmatter}
\title{Asymptotics of determinants of Hankel matrices via non-linear difference equations}

\author[ELB]{Estelle L. Basor}
\ead{ebasor@aimath.org}
\address[ELB]{American Institute of Mathematics,
360 Portage Avenue, Palo Alto, CA 94306-2244, USA}
\author[YC]{Yang Chen}
\ead{yayangchen@umac.mo}
\address[YC]{ Faculty of Science and Technology, Department of Mathematics, University of Macau, Av. Padre Tom\'as Pereira, Taipa Macau, China}
\author[NSH]{Nazmus S. Haq\corref{cor1}}
\ead{nazmus.haq04@imperial.ac.uk}
\address[NSH]{Department of Mathematics, Imperial College London,
180 Queen's Gate, London SW7 2BZ, UK}
\cortext[cor1]{Corresponding author}



\begin{abstract}
E. Heine in the 19th century studied a system of orthogonal polynomials associated with the weight $\left[x(x-\al)(x-\bt)\right]^{-\frac{1}{2}}$, $x\in[0,\al]$, $0<\al<\bt$. A related system
was studied by C. J. Rees in 1945, associated with the weight $\left[(1-x^2)(1-k^2x^2)\right]^{-\frac{1}{2}}$, $x\in[-1,1]$, $k^2\in(0,1)$. These are also known as
elliptic orthogonal polynomials, since the moments of the weights maybe expressed in terms of elliptic integrals.
Such orthogonal polynomials are of great interest because the corresponding Hankel determinant, depending on
a parameter $k^2$, where $0<k^2<1$ is the $\tau$~function of a particular Painlev\'e VI, the special cases of
which are related to enumerative problems arising from string theory. We show that the recurrence coefficients, denoted by
$\bt_n(k^2),\;n=1,2,\dots$; and $\P_1(n,k^2)$, the coefficients of
$x^{n-2}$ of the monic polynomials orthogonal with respect to a generalized version of the weight studied by Rees,
$$
(1-x^2)^{\al}(1-k^2x^2)^{\bt},\;\;x\in[-1,1],\;\alpha>-1,\;\bt\in \mathbb{R},
$$
satisfy second order non-linear difference equations. The large $n$ expansion based on the difference equations when
combined with known asymptotics  of the leading terms of the associated Hankel determinant yields a complete
asymptotic expansion of the Hankel determinant.  The Painlev\'e equation is also discussed as well as
the generalization of the linear second order differential equation found by Rees.
\end{abstract}



\end{frontmatter}



\numberwithin{equation}{section} 

\section{Introduction}
The study of Hankel determinants has seen a flurry of activity in recent years in part due to connections with Random Matrix theory (RMT).
This is because Hankel determinants compute the most fundamental objects studied in RMT. For example, the determinants may represent the
partition function for a particular random matrix ensemble, or they might be related to the distribution of the largest eigenvalue or they
may represent the generating function for a random variable associated to the ensemble.

Another recent interesting application of Hankel determinants is to compute  certain Hilbert series that are used to count the
 number of gauge invariant quantities on moduli spaces and  to characterize moduli spaces of a wide range of supersymmetric gauge theories.
Many aspects of supersymmetric gauge theories can be analyzed exactly, providing a ``laboratory" for the dynamic of gauge theories.
 For additional information about this topic, see \cite{ChenJJM2013,BasorChenMekar2012}. In these papers, heavy use are made of the mathematics
 involving two important types of matrices: Toeplitz and Hankel.

Often there is an associated Painlev\'e equation
that is satisfied by the logarithm of the Hankel determinant with respect to some parameter. This is true, for example,
in the Gaussian Unitary ensemble and for many other classical cases \cite{TracyWidom1999}.
In a recent development, one finds that Painlev\'e equations also appear in the information theoretic aspect of
wireless communication systems \cite{ChenMckay2010}.
Once the Painlev\'e equation is found,
then the Hankel determinant is much better understood. Asymptotics can be found via the Painlev\'e equation, scalings can be made to find limiting
densities, and in general the universal nature of the distributions can be analyzed. Other methods, including Riemann-Hilbert techniques and
general Fredholm operator theory methods, have also been used very successfully to find these asymptotics along with the Painlev\'e
equation analysis.

In this paper, where our approach is different, our focus is on the modified Jacobi weight,
$$
(1-x^2)^{\al}(1-k^2x^2)^{\bt},\;\;x\in[-1,1],\;\alpha>-1,\;\bt\in \mathbb{R}, \;k^2\in(0,1).
$$
We find asymptotics for the determinant, but our main technique is to compute these asymptotics from difference equations and then combine the information obtained from the difference equation with known asymptotics for  the leading order terms.
This is done by finding equations for auxiliary quantities defined by the corresponding orthogonal polynomials. The main idea is to use the very useful and practical ladder operator approach developed in \cite{ChenIsmail1997} and \cite{ChenIsmail2005}.
Recent applications of Riemann-Hilbert techniques on the asymptotics of orthogonal polynomials and Hankel determinants associated with similar modified Jacobi weights can be found in \cite{XuZhao_Jacobi,ZengXuZhao_Jacobi}. 
\subsection{Heine and Rees}
In the 19th century, Heine \cite{Heine1878}, considered polynomials orthogonal with respect to the weight,
\bea\label{def:w_Heine}
w_{\rm H}(x)&=&\frac{1}{\sqrt {x(x-\al)(x-\bt)}},\qquad x\in[0,\al],\;\; 0<\al<\bt,
\eea
and derived a second order ode satisfied by them \cite[p. 295]{Heine1878}. This ode is a generalization
of the hypergeometric equation, but of course not in the conventional ``eigenvalue-eigenfunction" form.

Rees \cite{Rees1945}, in 1945, studied a similar problem, with weight,
\bea\label{def:w_Rees}
w_{\rm R}(x)&:=&\frac{1}{{\sqrt {(1-x^2)(1-k^2x^2)}}},\qquad x\in[-1,1],\;\; k^2\in(0,1),
\eea
and used a method due to  Shohat \cite{Shohat1939}, essentially a variation of that employed by Heine,
to derive a second order ode. 
The ode obtained by Heine \cite[p. 295]{Heine1878} reads,
\bea\label{eq:R:HeineDiffEqn}
\quad 2x(x-\al)(x-\bt)(x-\ga)P^{\;\prime\prime}_n(x)
+ \bigg[(x-\ga)\frac{d}{dx}\left(\frac{1}{[w_{\rm H}(x)]^2}\right)-\frac{2}{[w_{\rm H}(x)]^2}\bigg]P_n^{\;\prime}(x)
\nonumber\\
+\Big[a+bx-n(2n-1)x^2\Big]P_n(x)&=&0.\quad
\eea
\noindent There are 3 parameters, $a,$ $b$ and $\gamma$  in Heine's differential equation (\ref{eq:R:HeineDiffEqn}).
Furthermore, $a$ and $b$  are expressed in terms of $\ga$ as roots of
2 algebraic equations, however $\ga$ is not characterized.
Therefore Heine's ode is to be regarded as an existence proof, and appeared not to be suitable for
the further study of such orthogonal polynomials.

For polynomials associated with $w_{\rm R}(x)$, Rees \cite[eq. 48]{Rees1945} derived the following second order ode:
\bea
\label{eq:R:ReesDiffEqn}
\quad \frac{M_n(x)}{w_{\rm R}(x)^2}P^{\;\prime\prime}_n(x)
+\bigg[\frac{M_n(x)}{2}\frac{d}{dx}\left(\frac{1}{w_{\rm R}(x)^2}\right)- \frac{M^{\prime}_n(x)}{w_{\rm R}(x)^2}\bigg]P^{\;\prime}_n(x)
+\bigg[L_n(x)M^{\prime}_n(x)+M_n(x)U_n(x)\bigg]P_n(x) &=& 0,
\qquad
\eea
where\footnote{We identify $\bt_n$ and $U_n(x)$ with Rees' $\la_{n+1}$ and $D_n(x)$ respectively.}
\bea
M_n(x)&=& -(2n-1)k^2x^2-(2n+1)k^2(\bt_n+\bt_{n+1})+2n(1+k^2)-4k^2\;\sum_{j=1}^{n-1}\bt_j,\\
L_n(x)&=& nk^2x^3+\bigg[(2n-1)k^2\bt_n-n(1+k^2)+2k^2\;\sum_{j=1}^{n-1}\bt_j\bigg]x,\\
U_n(x)&=&-n(n+1)k^2x^2-2(2n-1)k^2\;\sum_{j=1}^{n}\bt_j+n^2(1+k^2).
\eea
Rees found a difference equation \cite[eq. 55]{Rees1945} satisfied by $\bt_n$, and $\sum\limits_j\bt_j$ and so,
by specifying $\bt_0$, $\bt_1$ and $\bt_2$, it would be possible, at least in principle,
to determine all $\bt_n$ iteratively.

In this paper, we study a generalization of Rees' problem. Our polynomials are orthogonal with respect to
the following weight:
\begin{equation}\label{defn:R:w(xk2)}
w(x,k^2)=(1-x^2)^{\alpha}(1-k^2x^2)^{\beta},\qquad x\in [-1,1],\;\;\al>-1,\;\;\bt\in\mathbb{R},\;\; k^2\in(0,1).
\end{equation}
This weight may be regarded as a deformation of the Jacobi weight $w^{(\al,\al)}(x)$, where
\bea\label{def:JacobiWeight}
w^{(\al,\bt)}(x)&=&(1-x)^{\al}(1+x)^{\bt}, \qquad x\in[-1,1], \;\; \al>-1,\;\; \bt>-1,
\eea
with the ``extra" multiplicative factor $(1-k^2x^2)^{\bt}$.

If $\al=-\frac{1}{2}$ and $\bt=-\frac{1}{2}$, then (\ref{defn:R:w(xk2)}) reduces to  Rees' weight function (\ref{def:w_Rees}).

Instead of following the method employed by Heine and by Rees, we use the theory
 of ladder operators for orthogonal polynomials
 \cite{TracyWidom1999,Magnus1995,BasorChenEhrhardt,ChenIsmail2005,ChenIts2009} and the associated supplementary
 conditions $(S_1)$, $(S_2)$ and $(S_2^\prime)$ (described later) to find a set of  difference equations.
The Hankel determinant, generated by (\ref{defn:R:w(xk2)}), is defined as follows:
\bea\label{def:R:Hankel}
D_n[w(\cdot,k^2)] = \det\left( \mu_{i+j}(k^2)
\right)_{i,j=0}^{n-1},
\eea
where the moments are 
\bea\label{eq:R:moments}
\mu_{j}(k^2) := \int\limits_{-1}^{1}\! x^{j}
w(x,k^2)\, dx, \qquad j = 0, 1, 2,\dots.
\eea
We state here for future reference facts on orthogonal polynomials and Hankel determinants. These can be found
in Szeg\"o's treatise \cite{Szego1939}.

For a given weight, say, $w(x)$, an even function, supported on $[-A, A]$, which has infinitely many moments
\bea
\mu_j[w]:=\int\limits_{-A}^{A}\!x^{j}w(x)\,dx,\qquad j=0,1,2,\dots,
\eea
it is a classical result that the Hankel determinant, $D_n[w]:=\det(\mu_{j+k})_{ j, k=0}^{n-1}$, admits
the following multiple integral representation:
$$
D_n[w]=\frac{1}{n!}\int\limits_{[-A,A]^n}\!\prod_{1\leq j<k\leq n}(x_k-x_j)^2\prod_{\ell=1}^{n}w(x_\ell)\,dx_\ell.
$$
The monic polynomials orthogonal with respect to $w$ over $[-A,A]$
\bea\label{def:Pn}
P_n(x)=x^n+\P_1(n)x^{n-2}+\dots+P_n(0),
\eea
satisfies the orthogonality relations
\bea\label{def:hn}
\int\limits_{-A}^{A}\!w(x)P_j(x)P_k(x)\,dx=h_j\;\delta_{j,k}, \qquad j,k=0,1,2,\dots,
\eea
where $h_j$ is the square of the $L^2$ norm over $[-A,A]$.
\\
$D_n$ admits a further representation
\bea\label{def:prodhn}
D_n[w]=\prod_{j=0}^{n-1}h_j.
\eea
From the orthogonality relations, there follows the three-term recurrence relation
\bea\label{def:recurrence}
xP_n(x)=P_{n+1}(x)+\bt_n\;P_{n-1}(x),\qquad n = 0,1,2,\dots,
\eea
subjected to the initial conditions
$$
P_0(x):=1,\qquad\bt_{0}:=0, \qquad P_{-1}(x):=0.
$$
Here
\bea\label{def:betan}
\bt_n=\frac{h_{n}}{h_{n-1}}.
\eea
From (\ref{def:recurrence}) and (\ref{def:Pn}), an easy computation gives
\bea\label{def:p1nbetan}
\P_1(n)-\P_1(n+1)&=&\bt_n,
\eea
where $\P_1(0)=\P_1(1):=0$,
and consequently,
\bea\label{def:p1n}
\P_1(n)=-\sum_{j=0}^{n-1}\bt_j.
\eea
The recurrence coefficient $\bt_n$, in terms of $D_n,$ reads,
\bea\label{def:dnbtn}
\bt_n=\frac{D_{n+1}D_{n-1}}{D_n^2}.
\eea
For the sake to expedite the discussion in this paper, we have presented facts on orthogonal polynomials with even weight functions.

The main task of this paper is the computation of the large $n$ expansion for $D_n$, through two non-linear difference equations in $n$, one
second and the other third order, satisfied by $\bt_n$. We also find a second-order non-linear difference equation satisfied by $\P_1(n)$. These are derived through a systematic application of equations
$(S_1)$, $(S_2)$ and $(S_2^\prime)$ (to be presented later). The hard-to-come-by constants independent of $n$ in the leading term of the expansion and other relevant terms,
we shall obtain towards the end of this paper. They can actually be computed via three different methods, one from the equivalence of
our Hankel determinant
to the determinant of a Toeplitz$+$Hankel matrix, generated by a particular singular
weight, together with large $n$ asymptotics of the latter; and by two other more direct methods. A systematic large~$n$ expansion for
$D_n$ is then obtained by ``integrating" (\ref{def:dnbtn}).

It should be pointed out that the form of the asymptotic expansion for the Hankel determinant as well as for the coefficients in the monic orthogonal polynomials was given in \cite{KuijlaarsMcVaVan2004}. Thus for our situation, this guarantees that the expansion we give in Section \ref{Sec:R:LargeN_Exp} of this paper hold for $\beta_n$ and $\P_1(n)$. Our expansion for $D_n$ holds independently of \cite{KuijlaarsMcVaVan2004} since we rely on converting the determinant to one involving Toeplitz plus Hankel determinants where the results are known. Here we also have an explicitly determined constant.

We also find the analogue of the second order linear equation satisfied by the orthogonal polynomials found
 by Rees and show that the logarithmic derivative of our Hankel determinant is
 related to the $\sig$-form of a particular Painlev\'e VI differential equation. Similar results can be found in a
 straight-forward way for the Heine weight using the same technique, but we are not including them here.

 Here is an outline of the rest of the paper. In the next section we give a summary of results. In Section~\ref{Sec:R:LadOpAuxVar}, the ladder operator approach is used to find equations in the auxiliary variables.
 This leads directly to Sections \ref{Sec:R:DiffEqnProofs} and \ref{Sec:R:DiffEqn_btn_p1n} where the proofs of the difference equations are given and
 also the analogue of the derivation of the second-order ode. For the reader interested only in the results, these two sections can be omitted.

Section \ref{Sec:R:Special_Cases} is devoted to some special cases of the weight which reduce to the classical weight and this
section serves as a verification of the method. 
The heart of the computation for the Hankel determinant is Section~\ref{Sec:R:LargeN_Exp}, where the difference equation satisfied by $\bt_n$ is used
to compute the correction terms of the large $n$ expansion and then this is tied to Section~\ref{Sec:R:Hankel_LargeN}, where the leading order terms
are computed from known results. We provide an alternative computation for the Hankel determinant in Section~\ref{Sec:R:Hankel_LargeN_Toda}, where we combine the large $n$ expansion of $\P_1(n)$ (also obtained in Section~\ref{Sec:R:LargeN_Exp}) with `time-evolution' equations satisfied by the Hankel determinant, and integrate. The final section describes the Painlev\'e equation.

\subsection{Summary of results}

For polynomials orthogonal with respect to (\ref{def:w_Rees}), Rees derived in \cite[eq. 55]{Rees1945} the
following difference equation involving $\bt_{n+1}$, $\bt_n$, $\bt_{n-1}$, $\bt_{n-2}$, $\P_1(n)$ and $\P_1(n-2)$:
\bea\label{eq:R:Rees_D_Eqn}
\bt_{n-1}C^{\rm Rees}_{n-2}&=&\bt_n C_n^{\rm Rees}+1,
\eea
where\footnote{In Rees' paper \cite{Rees1945}, $C_n^{\rm Rees}$ is identified with $H_n$.}
\bea\label{eq:R:Cn(Rees)}
C_n^{\rm Rees}&:=&\left(2n+1\right)k^2(\bt_n+\bt_{n+1})-2n(k^2+1)-4k^2\P_1(n).
\eea

For the weight (\ref{defn:R:w(xk2)}), we derive through the use of the ladder operators and the associated
supplementary conditions, a quadratic equation in $\P_1(n)$ with coefficients in $\bt_{n+1}$, $\bt_n$, $\bt_{n-1}$.
See the theorem below:
\begin{theorem}\label{thm:R:p1(n)^2Eqn}
The recurrence coefficient $\bt_n$ and  $\P_1(n)$ satisfies the following difference equation:
\begin{small}
\bea\label{eq:R:D_Eqn_btn_p1n^2_Intro}
0&=&k^2[\P_1(n)]^2+\Bigg[2k^2\Big(\al+\bt+n-\frac{1}{2}\Big)\bt_n-\al k^2-\bt\Bigg]\P_1(n)-
k^2\Big(\al+\bt+n+\frac{3}{2}\Big)\Big(\al+\bt+n-\frac{1}{2}\Big)\bt_n^2\nonumber\\
&&-\Bigg[k^2\Big(\al+\bt+n+\frac{3}{2}\Big)\Big(\al+\bt+n-\frac{1}{2}\Big)\bt_{n+1}-
\bigg\{\Big(\bt+n+\frac{1}{2}\Big)k^2+\Big(\al+n+\frac{1}{2}\Big)\bigg\}\Big(\al+\bt+n-\frac{1}{2}\Big)\nonumber\\
&&\qquad+k^2\Big(\al+\bt+n+\frac{1}{2}\Big)\Big(\al+\bt+n-\frac{3}{2}\Big)\bt_{n-1}\Bigg]\bt_n-\frac{n}{2}\left(\frac{n}{2}+\al+\bt\right).
\eea
\end{small}\noindent
\end{theorem}
Solving for $\P_1(n)$ and noting the fact that $\P_1(n)-\P_1(n+1)=\bt_n$, 
a third order difference equation for $\bt_n$ is found. This is stated in the following theorem:
\begin{theorem}\label{thm:R:3rd_Ord_D_Eqn_btn}
$\bt_n$ satisfies the following third-order difference equation:
\begin{small}
\begin{multline}
\label{eq:R:3rd_Ord_D_Eqn_btn_Intro}
(\bt_{n+1}-\bt_n)^2\Bigg\{4k^4\bt_n\bigg[\Big(\al+\bt+n+\frac{3}{2}\Big)
\Big(\al+\bt+n-\frac{1}{2}\Big)\bt_{n+1}+\Big(\al+\bt+n+\frac{1}{2}\Big)\Big(\al+\bt+n-\frac{3}{2}\Big)\bt_{n-1}\bigg] \quad \quad \quad \quad \quad
\quad \quad \\
+4k^2\Big(\al+\bt+n+\frac{1}{2}\Big)\Big(\al+\bt+n-\frac{1}{2}\Big)\bt_n\Big(2k^2\bt_n-k^2-1\Big)
+(\al k^2+\bt)^2+k^2n(n+2\al+2\bt)\Bigg\}
=   \\
\Bigg\{\bt_{n+1}\bigg[k^2\Big(\al+\bt+n+\frac{5}{2}\Big)(\bt_{n+2}+\bt_{n+1})-(k^2+1)\Big(\al+\bt+n+\frac{3}{2}\Big)\bigg]
\\
-\bt_n\bigg[k^2\Big(\al+\bt+n+\frac{1}{2}\Big)(\bt_{n}+\bt_{n-1})-(k^2+1)\Big(\al+\bt+n-\frac{1}{2}\Big)\bigg]
+\frac{1}{2}+2k^2\bt_n\bigg[\Big(\al+\bt+n+\frac{1}{2}\Big)(\bt_{n+1}-\bt_n)+\bt_n+\bt_{n-1}\bigg]\Bigg\}^2.
\end{multline}
\end{small}
\end{theorem}\noindent
The above difference equation can also be regarded as a quartic equation in $\bt_{n}$ and $\bt_{n+1}$, and quadratic in $\bt_{n-1}$ and $\bt_{n+2}$. 

This highlights the advantage of our approach. 
If we were to eliminate $\P_1(n)$ in
(\ref{eq:R:Rees_D_Eqn}), we would find $\bt_n$ satisfies a fourth order difference equation.

Eliminating only the $[\P_1(n)]^2$ term in (\ref{eq:R:D_Eqn_btn_p1n^2_Intro}) and through the use of certain identities,
we obtain a generalization of (\ref{eq:R:Rees_D_Eqn}), valid for $\al>-1$, $\bt\in\mathbb{R}$.
\begin{theorem}\label{thm:R:Rees_General}
For $\al>-1$ and $\bt\in\mathbb{R},$ we have
\bea\label{eq:R:Gen_Rees_D_Eqn}
\bt_{n-1}C_{n-2}&=&\bt_n C_n+1,
\eea
where
\begin{small}
\bea\label{eq:R:Cn}
C_n&:=&2\left(\al+\bt+n+\frac{3}{2}\right)k^2(\bt_n+\bt_{n+1})
-2\left[\left(\bt+n+\frac{1}{2}\right)k^2+\left(\al+n+\frac{1}{2}\right)\right]-4k^2\P_1(n).
\eea
\end{small}
The equation (\ref{eq:R:Gen_Rees_D_Eqn}) reduces to Rees' equation (\ref{eq:R:Rees_D_Eqn}), if $\al=\bt=-1/2.$
\end{theorem}
We  present here a second order difference equation satisfied by $\bt_n$, and later give a proof
independent of Theorem \ref{thm:R:p1(n)^2Eqn} and Theorem \ref{thm:R:3rd_Ord_D_Eqn_btn}.
\begin{theorem}\label{thm:R:2nd_Ord_D_Eqn}
The recurrence coefficient $\bt_n$ satisfies a second order difference equation,
which turns out to be an algebraic equation of total degree $6$ in $\bt_{n+1}$, $\bt_n$ and $\bt_{n-1}$:
\bea\label{eq:R:2nd_Ord_Diff_Eqn}
\sum\limits_{p=0}^{3}\sum\limits_{q=0}^{6}\sum\limits_{r=0}^{3}\!{c}_{p,q,r}\,\bt_{n+1}^{p}\,\bt_n^{q}\,\bt_{n-1}^{r}&=&0,\qquad\qquad
p+q+r\leq 6.
\eea
The complete list of the 34 {\em non-zero} coefficients ${c}_{p,q,r}$ are presented in \ref{App:2ndOrd_DE_Coeffs}.
\end{theorem}
Using methods similar to the proof of Theorem \ref{thm:R:2nd_Ord_D_Eqn}, we obtain
 a {\em second order} difference equation satisfied by $\P_1(n)$, presented below:
\begin{theorem}\label{thm:R:2D_Eqn_p1n}
The coefficient of the sub-leading term of $P_n(x)$, $\P_1(n)$, satisfies the following second order non-linear difference equation:
\begin{scriptsize}
\bea\label{eq:R:2D_Eqn_p1n}
0&=&
{k}^4 \left(\alpha+\beta+n-\frac{3}{2}\right)\left(\alpha+\beta+n+\frac{1}{2}\right)^2
\bigg[
\P_{1}(n+1)^2\Big(\P_{1}(n)-\P_{1}(n-1)\Big ) - \P_{1}(n-1)^2\Big(\P_{1}(n)-\P_{1}(n+1)\Big)
\bigg]
\nonumber\\&&
+ \left(\alpha+\beta+n-\frac{1}{2}\right)^2
\bigg[
{k}^2 \left(\alpha+\beta+n-\frac{3}{2}\right) \P_{1}(n-1)
-{k}^2 \left(\alpha+\beta+n+\frac{1}{2}\right)\P_{1}(n+1)
-\left(\bt+n-\frac{1}{2}\right) {k}^2 -\al-n+\frac{1}{2}
\bigg]
 k^2\P_{1}(n)^2
\nonumber\\&&
+ \left(\alpha+\beta+n+\frac{1}{2}\right)\left(\alpha+\beta+n-\frac{3}{2}\right)
k^4\P_{1}(n+1) \P_{1}(n) \P_{1}(n-1)
\nonumber\\&&
+ k^2\left(\alpha+\beta+n+\frac{1}{2}\right)\left(\alpha+\beta+n-\frac{3}{2}\right)
\left[
\left(\beta+n-\frac{1}{2}\right)k^2+\al+n-\frac{1}{2}
\right]
\Big(\P_1(n+1)\P_1(n)+\P_1(n)\P_1(n-1)-\P_1(n+1)\P_1(n-1)\Big)
\nonumber\\&&
+ (\al k^2+\bt)
\Bigg[
\left(\alpha+\beta+n+\frac{1}{2}\right)\P_1(n+1)-
\left(\alpha+\beta+n-\frac{3}{2}\right)\P_1(n-1)
\Bigg]
k^2\P_1(n)
 \nonumber\\&&
 + \frac{(n-1)}{2}\left(\alpha+\beta+n+\frac{1}{2}\right)\left(\alpha+\beta+\frac{n}{2}-\frac{1}{2}\right) k^2\P_{1}(n+1)
- \frac{n}{2}\left(\alpha+\beta+n-\frac{3}{2}\right)\left(\alpha+\beta+\frac{n}{2}\right)
k^2\P_{1}(n-1)
\nonumber\\&&
+\Bigg[
\alpha \left(\beta+n-\frac{1}{2}\right){k}^4
+\frac{1}{2}\left(\alpha-\beta+n-\frac{1}{2}\right)\left(\alpha-\beta-n+\frac{1}{2}\right){k}^2 +\bt\left(\alpha+n-\frac{1}{2}\right)
\Bigg] \P_{1}(n)
+ \frac{1}{4}n(n-1) (\alpha {k}^2+\bt).
\eea
\end{scriptsize}

\end{theorem}

It is clear that,
$\P_1(n)$, will depend on $k^2$; but for brevity, we
do not display this dependence, unless required.

\begin{theorem}\label{thm:R:Diff_Eqn_Pn}
The orthogonal polynomials $P_{n}(x)$ with respect to the weight (\ref{defn:R:w(xk2)}) satisfy
\bea\label{eq:R:DiffEqnPn_Gen}
P_n^{\;\prime\prime}(x)+\left(\frac{1}{2}\frac{X^{\prime}(x)}{X(x)}
-\frac{M_n^\prime(x)}{M_n(x)}\right)P_n^{\;\prime}(x)+\left(\frac{L_n(x)M_n^\prime(x)}{Y(x)M_n(x)}+\frac{U_n(x)}{Y(x)}\right)P_n(x)&=&0,
\eea
where
\bea
X(x)&:=&(1-x^2)^{2\al+2}(1-k^2x^2)^{2\bt+2},\\
Y(x)&:=&(1-x^2)(1-k^2x^2),\\
M_n(x)&:=&-2\left(\al+\bt+n+\frac{1}{2}\right)k^2x^2-C_n,\\
L_n(x)&:=&x\Big[nk^2x^2-n(1+k^2)+2k^2\left(\al+\bt+n+\frac{1}{2}\right)\bt_n-2k^2\P_1(n)\Big],\\
U_n(x)&:=&-k^2x^2n(n+2\al+2\bt+3) + 2k^2(2n+2\al+2\bt+1)(\P_1(n)-\bt_n)+nk^2(n+2\bt+1)\nonumber\\&&
+n(n+2\al+1),
\eea
and $C_n$ is given by (\ref{eq:R:Cn}).
\end{theorem}

With the choice of $\alpha = -1/2$, $\beta = -1/2$ this reduces to the original equation of Rees.

The latter sections of this paper are then devoted to obtaining the following result, which is calculated using two alternative methods in
Sections \ref{Sec:R:Hankel_LargeN} and \ref{Sec:R:Hankel_LargeN_Toda} respectively.
\begin{theorem}\label{thm:R:Dn_LargeNExp_Intro}
The Hankel determinant $D_n[w(\cdot,k^2)]$ has the following asymptotic expansion in $n$:
\bea\label{eq:R:Dn(k2)_Large_n_Intro}
D_n[w(\cdot,k^2)]
&=&
E\;n^{\al^2-1/4}\; 2^{-n(n+2\al)}\;(2\pi)^{n}\;\left(\frac{1+\sqrt{1-k^2}}{2}\right)^{2\bt n}
\nonumber\\*&&
\times\;\exp
\Bigg[
\frac{2a_3}{n}
+\frac{2a_4-a_2\Big(4a_2+1\Big)}{3n^2}
+\frac{a_5-a_3\Big(4a_2+1\Big)}{3n^3}+{\rm O}\left(\frac{1}{n^4}\right)
\Bigg],
\eea
where the $n$-independent constant $E$ is given by
\bea
E&:=&
\frac{(2\pi)^{\al}\pi^{1/2}[G(1/2)]^2}{2^{2(\al^2+\al\bt+\bt^2)}[G(1+\al)]^2}\cdot\frac{\left(1+\sqrt{1-k^2}\right)^{2\bt(\al+\bt)}}{(1-k^2)^{\bt(\al+\bt/2)}}.
\eea
In the above, the function $G$ refers to the Barnes $G$-function, an entire function that satisfies the recurrence relation
$G(z+1)=\Gamma(z)G(z)$, $G(1):=1$; and the coefficients $a_2(k^2,\alpha,\beta)$---$a_5(k^2,\alpha,\beta)$ are the large $n$ expansion coefficients for $\bt_n$, which can be found in Section \ref{sec:R:2_D_btn_Ln_Exp}.
\end{theorem}

Finally we present the desired Painlev\'e equation for the logarithmic derivative of the Hankel determinant with respect to $k^2$. This equation is not new, but follows from a change of variables and the equation found in \cite{DaiZhang2010}.
\begin{theorem}\label{thm:R:PVI_Rep}
Let  $H_{2n}(k^2)$ and $H_{2n+1}(k^2)$ be defined via the Hankel determinant associated with $w(x,k^2),$ as
\bea
H_{2n}(k^2)&:=&k^2(k^2-1)\frac{d}{dk^2}\log D_{2n}[w(\cdot,k^2)],\\
H_{2n+1}(k^2)&:=&k^2(k^2-1)\frac{d}{dk^2}\log D_{2n+1}[w(\cdot,k^2)].
\eea
The functions $H_{2n}(k^2)$ and $H_{2n+1}(k^2)$ are then expressed in terms of the $\sig$-function of a Painlev\'e VI as follows:
\bea
\label{eq:R:H2n_PVI_Rep}
H_{2n}(k^2)&=&
\sig(k^2,n,-1/2,\al,\bt)+\sig(k^2,n,1/2,\al,\bt)
\nonumber\\&&
+\Big(\frac{\bt^2}{2}+2n\bt+\frac{1}{8}\Big)k^2
-\frac{\bt}{2}(2n+\al+\bt)+n(n+\al),\\
\label{eq:R:H2np1_PVI_Rep}
H_{2n+1}(k^2)&=&
\sig(k^2,n+1,-1/2,\al,\bt)+\sig(k^2,n,1/2,\al,\bt)
\nonumber\\&&
+\Big(\frac{\bt^2}{2}+(2n+1)\bt+\frac{1}{8}\Big)k^2
-\frac{\bt}{2}(2n+1+\al+\bt)+\frac{1}{4}(2n+1)(2n+1+2\al).
\eea
Here the function $\sig(k^2,n,a,b,c),$ as shown in \cite{DaiZhang2010}, satisfies
the Jimbo-Miwa-Okamoto $\sig$-form of a particular Painlev\'e VI  \cite{JimboMiwa1981vII}:
\bea
\sigma^\prime\Big(k^2(k^2-1)\sigma^{\prime\prime}\Big)^2+\Big\{2\sigma^\prime\big(k^2\sigma^\prime-\sigma\big)-
\big(\sigma^\prime\big)^2-\nu_1\nu_2\nu_3\nu_4\Big\}^2=
\prod\limits_{i=1}^{4}\big(\nu_i^2+\sigma^\prime\big),
\eea
where ${}^\prime$ denotes derivative with respect to $k^2$ and
\bea
\nu_1=\frac{1}{2}(c-a),\quad\nu_2=\frac{1}{2}(c+a),\quad\nu_3=\frac{1}{2}(2n+a+c),\quad\nu_4=\frac{1}{2}(2n+a+2b+c).
\eea
\end{theorem}

\begin{remark} The moments $\mu_j(k^2)$, $j=0,1,2,\dots$, of the weight (\ref{defn:R:w(xk2)})
can be expressed in terms of hypergeometric functions as
\bea\label{eq:R:moments(even)}
\mu_{2j}(\al,\bt,k^2)&=&
\frac{\Gamma(j+1/2)\Gamma(\al+1)}{\Gamma(j+\al+3/2)}
\ {}_2F_1\left(-\bt,j+\frac{1}{2};j+\al+\frac{3}{2};k^2\right),\\
\mu_{2j+1}(\al,\bt,k^2)&=&0,
\eea
where the hypergeometric function ${}_2F_1(a,b;c,z)$ has the following integral representation \cite[p. 558]{AbramowitzStegun}
\bea
{}_2F_1(a,b;c,z)&=&
\frac{\Gamma(c)}{\Gamma(b)\Gamma(c-b)}\int\limits^1_0\! t^{b-1}(1-t)^{c-b-1}(1-tz)^{-a}\,dt,\qquad \Re(c)>\Re(b)>0.
\nonumber
\eea
%
\end{remark}
%
%
%
%
\subsection{Ladder operators and supplementary conditions}\label{sec:Ladder}
In the theory of Hermitian
random matrices, orthogonal polynomials play an important role,
since the fundamental object, namely, Hankel determinants or
partition functions, are expressed in terms of the associated square of the $L^2$
norm, $D_n=\prod\limits_{j=0}^{n-1}h_j$.  Moreover, as we have seen in the introduction,
Hankel determinants are intimately related to
the recurrence coefficient $\bt_n$ of the orthogonal
polynomials (for other recent examples, see \cite{BasorChen2009,BasorChenEhrhardt,ChenZhang2010,ForresterOrmerod2010}).

There is a recursive algorithm that facilitates the
determination of the recurrence coefficient $\beta_n$. This is implemented through the use of so-called
``ladder operators'' as well as their associated supplementary conditions.
This approach can be traced back to Laguerre and Shohat \cite{Shohat1939}.
More recently, Magnus \cite{Magnus1995} applied this approach to
semi-classical orthogonal polynomials 
and the derivation of Painlev\'e equations, while
\cite{TracyWidom1999} used the  compatibility conditions
in the study of finite $n$ matrix models. See
\cite{BasorChenEhrhardt,ChenIsmail2005,ChenIts2009} for other examples.

With the potential, $\textsf{v},$ defined by
\bea\label{Defn:w=e^-vx}
\textsf{v}(x)=-\log w(x),
\eea
a pair of ladder operators---formulae that allow one to raise or lower the index $n$ of the orthogonal polynomial---may be
 derived \cite{ChenIsmail1997,ChenIsmail2005}:
\begin{equation}\label{Defn:LadderOp}\begin{aligned}
\left[\frac{d}{dx}+B_n(x)\right]P_n(x)&=\bt_nA_n(x)P_{n-1}(x),\\
\left[\frac{d}{dx}-B_n(x)-\textsf{v}^\prime(x)\right]P_{n-1}(x)&=-A_{n-1}(x)P_n(x),
\end{aligned}\end{equation}
where
\begin{equation}\label{Defn:A_nB_n}
\begin{aligned}
A_n(x)&=\frac{1}{h_n}\int\limits_{-1}^1\!\frac{\textsf{v}^\prime(x)-\textsf{v}^\prime(y)}{x-y}P_n^2(y)w(y)\,dy,\\
B_n(x)&=\frac{1}{h_{n-1}}\int\limits_{-1}^1\!\frac{\textsf{v}^\prime(x)-\textsf{v}^\prime(y)}{x-y}P_n(y)P_{n-1}(y)w(y)\,dy.
\end{aligned}
\end{equation}
The fundamental compatibility conditions for $A_n(x)$ and $B_n(x)$  are
\begin{equation}\tag{$S_1$}
B_{n+1}(x) + B_n (x) = xA_n(x) - \textsf{v}^\prime(x),
\end{equation}
\begin{equation}\tag{$S_2$}
1+x[B_{n+1}(x)-B_n(x)]=\beta_{n+1}A_{n+1}(x) - \beta_n A_{n-1}(x).
\end{equation}
See \cite{ChenIsmail1997,ChenIsmail2005} for a derivation.
These were initially derived for any polynomial $\textsf{v}(x)$ (see
\cite{Bauldry1990,BonanClark1990,Mhaskar1990}), and then were shown
to hold for all $x\in\mathbb{C}\cup\{\infty\}$ in greater generality
\cite{ChenIsmail2005}.

We now combine $(S_1)$ and $(S_2)$ as follows.  First, multiply
$(S_2)$ by $A_n(x)$, then it is seen that the right side is a first order
difference, while $xA_n(x)$ on the left side is replaced by
$B_{n+1}(x)+B_n(x)+\textsf{v}^\prime(x)$ taking into account $(S_1)$. Now, taking
a telescopic sum with initial conditions
\begin{equation*}
B_0(x)=A_{-1}(x):=0,
\end{equation*}
leads to the important identity
\begin{equation}\tag{$S_2^{\ \prime}$}
\sum\limits_{j=0}^{n-1}\!A_j(x)\ + \ B_n^{\ 2}(x) +\textsf{v}^\prime
(x)B_{n}(x) =  \beta_n A_n(x)A_{n-1}(x).
\end{equation}
The condition $(S_2')$ is of considerable interest, since it is
intimately related to the logarithm of the Hankel determinant. In
order to gain further insight into the determinant, we need to
find a way to reduce the sum to a fixed number of quantities, known as the auxiliary variables
(to be introduced in the next section). The equation $(S_2^\prime)$ ultimately paves a way going forward.


\section{Computation of auxiliary variables}\label{Sec:R:LadOpAuxVar}

Starting with our weight function (\ref{defn:R:w(xk2)}),
\bea\label{eq:R:w(x,k2)}
w(x,k^2)&=&(1-x^2)^{\alpha}(1-k^2x^2)^{\beta}, \qquad x\in[-1,1],\;\;\al>-1,\;\;\bt\in\mathbb{R},\;\; k^2\in(0,1),
\eea
we see that
\bea
\label{eq:R:v(x)}\textsf{v}(x)&=&-\log w(x,k^2)= -\alpha\log(1-x^2)-\bt\log(1-k^2x^2),
\eea
and
\bea
\label{eq:R:vp(x)}\textsf{v}^\prime(x)&=& -\alpha\left(\frac{1}{x-1}+\frac{1}{x+1}\right)-
\beta\left(\frac{1}{x-1/k}+\frac{1}{x+1/k}\right).
\eea
\begin{remark}
Since $\textsf{v}^\prime(x)$ is a rational function of $x$,
\begin{small}
\bea
\label{eq:R:vpx-vpy}
\frac{\textsf{v}^\prime(x) -\textsf{v}^\prime(y)}{x-y}
=\frac{\alpha}{x-1}\cdot\frac{1}{y-1}+\frac{\alpha}{x+1}\cdot\frac{1}{y+1}
+\frac{\beta}{x-1/k}\cdot\frac{1}{y-1/k}+\frac{\beta}{x+1/k}\cdot\frac{1}{y+1/k}
\eea
\end{small}\noindent
is also a rational function of $x,$ and $y,$ which in turn implies that
$A_n(x)$ and $B_n(x)$ are rational functions of~$x$. Consequently,
equating the residues of the simple and double pole at  $x=\pm1$, and $x=\pm1/k$, and
the limit at $x\to\infty$ on both sides of the supplementary conditions
$(S_1)$, $(S_2)$ and $(S_2^\prime)$, we obtain non-linear difference equations in $n$ satisfied
by the auxiliary variables.  These equations are likely very complicated, but the
main idea is to express the recurrence coefficient $\bt_n$ in terms of these auxiliary variables, and eventually take
advantage of the product representation $D_n=\prod\limits_{j=0}^{n-1}h_j$  to obtain
equations satisfied by the logarithmic derivative of the Hankel determinant.
\end{remark}
Using $P_n(-y)=(-1)^nP_n(y),$ we may deduce that 
\bea
\label{eq:R:Sym_rel_1}
\int\limits^1_{-1}\!\frac{P_n^2(y)w(y)}{c-y}\,dy
&=&\int\limits^1_{-1}\!\frac{P_n^2(y)w(y)}{c+y}\,dy
,\\
\label{eq:R:Sym_rel_3}
\int\limits^1_{-1}\!\frac{P_n(y)P_{n-1}(y)w(y)}{c-y}\,dy&=&
-\int\limits^1_{-1}\!\frac{P_n(y)P_{n-1}(y)w(y)}{c+y}\,dy,
\eea
where $c\in\mathbb{R}$ is fixed. From the above elementary relations, 
and
(\ref{Defn:A_nB_n}), we find
\bea
\label{eq:R:An}A_n(x)&=&-\frac{R_n}{x-1}+\frac{R_n}{x+1}-\frac{R_n^\ast}{x-1/k}+\frac{R_n^\ast}{x+1/k},\\
\label{eq:R:Bn}B_n(x)&=&\frac{r_n}{x-1}+\frac{r_n}{x+1}+\frac{r_n^\ast}{x-1/k}+\frac{r_n^\ast}{x+1/k},
\eea
where
\begin{equation}\label{eq:R:AuxVar}\begin{aligned}
R_n^\ast(k^2):=&\frac{\beta}{h_{n}}\int\limits^1_{-1}\!\frac{w(y)P_n^2(y)}{1/k+y}\,dy,&&&&&&&&
r_n^\ast(k^2):=&\frac{\beta}{h_{n-1}}\int\limits^1_{-1}\!\frac{w(y)P_n(y)P_{n-1}(y)}{1/k+y}\,dy,\\
R_n(k^2):=&\frac{\alpha}{h_{n}}\int\limits^1_{-1}\!\frac{w(y)P_n^2(y)}{1+y}\,dy,&&&&&&&&
r_n(k^2):=&\frac{\alpha}{h_{n-1}}\int\limits^1_{-1}\!\frac{w(y)P_n(y)P_{n-1}(y)}{1+y}\,dy,
\end{aligned}\end{equation}
are the $4$ auxiliary variables.
%

\begin{remark}
Take $n=0$. From (\ref{eq:R:moments(even)}) and the definitions of $R_0(k^2)$, $R_0^\ast(k^2)$, $r_0(k^2)$ and
$r_0^\ast(k^2)$, it follows that
\bea
\label{eq:R:IC:R0}R_0(k^2)&=&\frac{\left(\al+\frac{1}{2}\right){}_2F_1\left(-\bt,\frac{1}{2};\al+\frac{1}{2};k^2\right)}{{}_2F_1\left(-\bt,\frac{1}{2};\al+\frac{3}{2};k^2\right)},\\
\label{eq:R:IC:R0ast}\frac{R_0^\ast(k^2)}{k}&=&\frac{\bt{}_2F_1\left(-\bt+1,\frac{1}{2};\al+\frac{3}{2};k^2\right)}{{}_2F_1\left(-\bt,\frac{1}{2};\al+\frac{3}{2};k^2\right)},\\
\label{eq:R:IC:r0r0ast}r_0(k^2)&=&r_0^\ast(k^2)=0.
\eea
Thus, $(S_1)$ and $(S_2)$ determine the sequence $A_{n}(x)$ and $B_{n}(x)$ given initial conditions $B_{0}(x)=A_{-1}(x)=0$ and the value of $A_{0}(x)$ from above.
\end{remark}
\subsection{Difference equations}
Inserting $A_n(x)$ and $B_n(x)$ into
$(S_1)$, $(S_2)$ and $(S_2^\prime)$,
and equating the residues yields a system of $17$~equations. However, there are multiple instances of the same equation, and thus we actually have a system of $9$ distinct equations.

From $(S_1)$, equating the residues at $x=\pm1$ and $x=\pm1/k$ give the following set of equations:
\bea
\label{sys:R:MCS1.1}r_n +r_{n+1}&=&\alpha-R_n,\\
\label{sys:R:MCS1.2} r_n^\ast+r_{n+1}^\ast&=&\beta-R_n^\ast/k.
\eea
Similarly, from $(S_2)$, the limit at $x\to\infty$, and equating the residues at $x=\pm1$ and $x=\pm1/k$ give the
following set of distinct equations:
\bea
\label{sys:R:MCS2.0}1+2r_{n+1}^\ast-2r_n^\ast+2r_{n+1}-2r_n&=&0,\\
\label{sys:R:MCS2.1} r_{n}-r_{n+1} &=&\beta_{n+1}R_{n+1}-\beta_{n}R_{n-1},\\
\label{sys:R:MCS2.2}(r_{n}^\ast-r_{n+1}^\ast)/k&=&\beta_{n+1}R_{n+1}^\ast-\beta_nR_{n-1}^\ast.
\eea
Going through the steps with $(S_2^\prime)$ is more complicated, but
equating all respective residues in $(S_2^\prime)$ yields four equations.
The first two  are obtained by equating the residues of the double
pole at $x=\pm1$, and $x=\pm1/k$ respectively:
\bea
\label{sys:R:MCS2p1.1}
r_n(r_n-\al)&=&
\bt_nR_nR_{n-1},\\
\label{sys:R:MCS2p1.2} r_n^\ast(r_n^\ast-\bt)&=&\beta_nR_n^\ast R_{n-1}^\ast,
\eea
while
the last two distinct equations are obtained by equating the residues of the simple pole at
$x=\pm1$ and $x=\pm1/k$ respectively:
\bea\label{sys:R:MCS2p2.1}
\frac{1}{2}\sum\limits_{j=0}^{n-1}R_j -\frac{k^2}{k^2-1}
\bigg(2r_n^\ast r_n-\al r_n^\ast-\bt r_n\bigg)
&=&
\beta_n R_nR_{n-1}-\frac{k\bt_n}{k^2-1}
\bigg(R_nR_{n-1}^\ast+ R_n^\ast R_{n-1}\bigg),\\*
\label{sys:R:MCS2p2.2}
\frac{1}{2}\sum\limits_{j=0}^{n-1}R_j^\ast +\frac{k}{k^2-1}
\bigg(2r_n^\ast r_n-\al r_n^\ast-\bt r_n\bigg)
&=&
k\beta_nR_n^\ast R_{n-1}^\ast +
\frac{k^2}{k^2-1}\beta_n\bigg( R_nR_{n-1}^\ast+ R_n^\ast R_{n-1}\bigg).
\eea
\begin{remark}
Since $(S_2^\prime)$ is obtained from $(S_1)$ and $(S_2)$, we can expect that (\ref{sys:R:MCS2p1.1})--(\ref{sys:R:MCS2p2.2}) can be derived from (\ref{sys:R:MCS1.1})--(\ref{sys:R:MCS2.2}) together with initial conditions. For example, (\ref{sys:R:MCS2p1.1}) can be obtained by multiplying (\ref{sys:R:MCS1.1}) and (\ref{sys:R:MCS2.1}), and then taking a telescopic sum of the resultant equation from $j=0$ to $j=n-1$ along with initial conditions. However, this approach can be somewhat unsystematic, compared to the more direct approach of using $(S_2^\prime)$.
\end{remark}
\subsection{Analysis of non-linear system}
While the difference equations
(\ref{sys:R:MCS1.1})--(\ref{sys:R:MCS2p2.2}) look rather complicated, our
aim is to manipulate these equations in such a way to give us
insight into the recurrence coefficient $\bt_n.$  Our
dominant strategy is to always try to describe the recurrence
coefficient $\bt_n$ in terms of the auxiliary variables
$R_n$, $r_n$, $R_n^\ast$ and $r_n^\ast$.

Taking a telescopic sum of (\ref{sys:R:MCS2.0}) from $j=0$ to $j=n-1$  while noting the
initial conditions $r_0=0$ and $r_0^\ast=0,$ 
we obtain,
\bea\label{eq:R:rnrnast}
r_n^\ast+r_n&=&-\frac{n}{2}.
\eea
Note that the above can also be obtained by
integration by parts on the definition of $r_n$ or $r_n^\ast$, (\ref{eq:R:AuxVar}).

From $(S_1)$, the sum of (\ref{sys:R:MCS1.1}) and (\ref{sys:R:MCS1.2}), then gives
\bea
\label{eq:R:RnRnastv1} \frac{R_n^\ast}{k}+R_n&=&\al+\bt+n+\frac{1}{2}.
\eea
Now summing the above from $j=0$ to $j=n-1$, gives
\bea
\label{eq:R:SumRjRjast}\sum\limits_{j=0}^{n-1}\!\left(\frac{R_j^\ast}{k}+R_j\right)=n\left(\alpha+\beta+\frac{n}{2}\right).
\eea

We now turn to $(S_2)$. The sum of (\ref{sys:R:MCS2.1}) and (\ref{sys:R:MCS2.2}), and
with (\ref{eq:R:RnRnastv1}) and (\ref{eq:R:rnrnast}) to eliminate $R_n^\ast$ and $r_n^\ast$ respectively, gives us,
\bea
\label{eq:R:S2Combined}(k^2-1)(r_n-r_{n+1})+\frac{1}{2}&=&k^2\left(\al+\bt+n+\frac{3}{2}\right)\beta_{n+1}
-k^2\left(\al+\bt+n-\frac{1}{2}\right)\beta_n.
\eea
Taking a telescopic sum of (\ref{eq:R:S2Combined}) from $j=0$ to $j=n-1$ and recalling
$\P_1(n)=-\sum\limits_{j=0}^{n-1}\bt_j$ and $\bt_0=0$  yields
\bea
\label{eq:R:p1n(rnbtn)}(k^2-1)r_n&=&\frac{n}{2}-k^2\left(\al+\bt+n+\frac{1}{2}\right)\beta_n+k^2\textsf{p}_1(n).
\eea
We are now in a position to derive an important lemma, which expresses the recurrence
coefficients $\bt_n$ in terms of
$r_n$ and $R_n$.
\begin{lemma}
The quantity $\bt_n$ can be expressed in terms of the auxiliary variables $r_n$ and $R_n$ as
\begin{small}
\bea
\label{eq:R:btn(Rnrn)}k^2\beta_n&=&
\Bigg\{\Bigg[1+k^2\left(\frac{\al+\bt+n+\frac{1}{2}}{R_n}-1\right)\Bigg]r_n^{\ 2}+
\Bigg[n+\beta - \alpha k^2\left(\frac{\al+\bt+n+\frac{1}{2}}{R_n}-1\right)\Bigg]r_n+
\frac{n}{2}\left(\frac{n}{2}+\beta\right)\Bigg\}\nonumber\\
&&\times\frac{1}{\Big(\al+\bt+n-\frac{1}{2}\Big)\Big(\al+\bt+n+\frac{1}{2}-R_n\Big)}.
\eea
\end{small}
\end{lemma}
\begin{proof}
Eliminating $R_n^\ast$ and $r_n^\ast$ from
(\ref{sys:R:MCS2p1.2}) using (\ref{eq:R:RnRnastv1}) and (\ref{eq:R:rnrnast}) respectively leads to
\bea\label{eq:R:btn(cn-Rn)(cnm1-Rnm1)}
k^2\bt_n\Big(\al+\bt+n+\frac{1}{2}-R_n\Big)\Big(\al+\bt+n-\frac{1}{2}-R_{n-1}\Big)
&=&\left(r_n+\frac{n}{2}\right)\left(r_n+\frac{n}{2}+\bt\right).
\eea
The result then follows from eliminating $R_{n-1}$ using (\ref{sys:R:MCS2p1.1}).
\end{proof}
Finally, we give an important identity expressing the sum $\sum\limits_{j=0}^{n-1} R_j$
in terms of $\bt_n$ and $r_n$, presented 
below.

First note the following lemma:
\begin{lemma}\label{lem:R:cnRn-1_cn-1Rn(btnrn)}
\bea\label{eq:R:cnRn-1_cn-1Rn(btnrn)}
(k^2-1)r_n^2-r_n(n+\bt+\al k^2)-\frac{n}{2}\left(\frac{n}{2}+\bt\right)
&=&k^2\bt_n\Bigg[\left(\al+\bt+n+\frac{1}{2}\right)R_{n-1}+\left(\al+\bt+n-\frac{1}{2}\right)R_n
\nonumber\\&&
-\left(\al+\bt+n+\frac{1}{2}\right)\left(\al+\bt+n-\frac{1}{2}\right)\Bigg].
\eea
\end{lemma}

\begin{proof}
This is obtained from (\ref{eq:R:btn(cn-Rn)(cnm1-Rnm1)}) and noting that
$\bt_nR_nR_{n-1}=r_n(r_n-\al)$.  
\end{proof}

\begin{theorem}\label{lem:R:SumR(btnrn)}
The sum $\sum\limits_{j=0}^{n-1} R_j$ may be expressed in terms of $\bt_n$ and $r_n$ as
\bea\label{eq:R:SumR(btnrn)}
\frac{1}{2}\sum_{j=0}^{n-1}R_j
&=&-r_n(\alpha+\beta+n)+\frac{n}{2}\left(\frac{n}{2}+\beta+\alpha k^2\right)\frac{1}{k^2-1} 
-\frac{k^2}{k^2-1}\left(\al+\bt+n+\frac{1}{2}\right)\left(\al+\bt+n-\frac{1}{2}\right)\bt_n.
\nonumber\\
\eea
\end{theorem}
\begin{proof}
From equation (\ref{sys:R:MCS2p2.1}), we first eliminate $R_n^\ast$ and $R_{n-1}^\ast$ in favor of
$R_n$ and $R_{n-1}$ using (\ref{eq:R:RnRnastv1}), and eliminate $r_n^\ast$ in favor of $r_n$ using
(\ref{eq:R:rnrnast}). With  $\bt_nR_nR_{n-1}=r_n(r_n-\al),$ 
we arrive at
\begin{small}
\bea
\frac{1}{2}\sum\limits_{j=0}^{n-1}\!R_j
&=&
r_n^{ 2}-r_n\left(\alpha+\frac{(\alpha+\beta+n)k^2}{k^2-1}\right)
+\frac{n\al k^2}{2(k^2-1)}
-\frac{k^2}{k^2-1}\beta_n\Bigg[\left(\al+\bt+n+\frac{1}{2}\right)R_{n-1}+\left(\al+\bt+n-\frac{1}{2}\right)R_n\Bigg].
\nonumber\\
\eea
\end{small}\noindent
Finally, we use equation (\ref{eq:R:cnRn-1_cn-1Rn(btnrn)}), 
to eliminate $$\left(\al+\bt+n+\frac{1}{2}\right)R_{n-1} + \left(\al+\bt+n-\frac{1}{2}\right)R_{n}$$ from the above
equation to arrive at our result.
\end{proof}

\section{Non-linear difference equation for \texorpdfstring{$\bt_n$}{recurrence coefficient}}\label{Sec:R:DiffEqnProofs}
\subsection{Proof of Theorem \ref{thm:R:p1(n)^2Eqn}}
We now prove Theorem \ref{thm:R:p1(n)^2Eqn}.

\begin{proof}
Eliminating $\bt_nR_nR_{n-1}$ from equation (\ref{eq:R:btn(cn-Rn)(cnm1-Rnm1)}) using (\ref{sys:R:MCS2p1.1}),
 leads to the following:
\bea\label{eq:R:Rec_Rn_Rnm1_rn}
k^2\bt_n\bigg[\Big(\al+\bt+n+\frac{1}{2}\Big)R_{n-1}+\Big(\al+\bt+n-\frac{1}{2}\Big)R_n\bigg]&=&k^2\Big(\al+\bt+n+\frac{1}{2}\Big)
\Big(\al+\bt+n-\frac{1}{2}\Big)\bt_n\nonumber\\
&&+k^2r_n(r_n-\al)-\left(r_n+\frac{n}{2}\right)\left(r_n+\frac{n}{2}+\bt\right).\nonumber\\
\eea
Our aim is to replace $R_n$, $R_{n-1}$ and $r_n$ in (\ref{eq:R:Rec_Rn_Rnm1_rn}) with $\bt_n$ and $\P_1(n)$.
We rearrange (\ref{eq:R:p1n(rnbtn)}) to express $r_n$ in terms of $\bt_n$ and $\P_1(n)$ as
\bea
\label{eq:R:rn(bti)}r_n&=&\frac{1}{k^2-1}\left[\frac{n}{2}-k^2\Big(\al+\bt+n+\frac{1}{2}\Big)\beta_n+k^2\textsf{p}_1(n)\right].
\eea
Replacing $n$ by $n-1$ and $n$ by $n+1$, in (\ref{eq:R:p1n(rnbtn)}), we find,
\bea
\label{eq:R:rnm1(bti)}r_{n-1}&=&\frac{1}{k^2-1}\left[\frac{n-1}{2}-k^2\Big(\al+\bt+n-\frac{3}{2}\Big)\beta_{n-1}+k^2\P_1(n)\right],
\eea
and
\bea
\label{eq:R:rnp1(bti)}r_{n+1}&=&
\frac{1}{k^2-1}\left[\frac{n+1}{2}-k^2\Big(\al+\bt+n+\frac{3}{2}\Big)\beta_{n+1}+k^2\big(\P_1(n)-\bt_n\big)\right],
\eea
respectively,
and we have used $\P_1(n-1)=\bt_{n-1}+\P_1(n)$ and $\P_1(n+1)=\P_1(n)-\bt_n$ to bring the right-hand sides into the above final
form.

Now we eliminate $R_n$ and $R_{n-1}$ in (\ref{eq:R:Rec_Rn_Rnm1_rn}).  First, we rearrange (\ref{sys:R:MCS1.1})
to give $$R_n=\al-r_n-r_{n+1}.$$
Substituting $r_n$ and $r_{n+1}$ given by (\ref{eq:R:rn(bti)}) and (\ref{eq:R:rnp1(bti)}) respectively into the
above equation, we see that $R_n$ may be expressed in terms of $\bt_n$ and $\P_1(n)$ as
\bea\label{eq:R:Rn(btnp1n)}
R_n&=&
\al-
\frac{1}{k^2-1}\left[n+\frac{1}{2}-k^2\Big(\al+\bt+n+\frac{3}{2}\Big)(\beta_n+\bt_{n+1})+2k^2\textsf{p}_1(n)\right].
\eea
Replacing  $n$ by $n-1$ and  $\P_1(n-1)$ by $\P_1(n)+\bt_{n-1}$ in (\ref{eq:R:Rn(btnp1n)}) we have 
\bea\label{eq:R:Rnm1(btnp1n)}
R_{n-1}&=&
\al-
\frac{1}{k^2-1}\left[n-\frac{1}{2}-k^2\Big(\al+\bt+n+\frac{1}{2}\Big)\beta_n-k^2\Big(\al+\bt+n-\frac{3}{2}\Big)
\bt_{n-1}+2k^2\P_1(n)\right]. 
\eea
In the final step we substitute $R_n$ given by (\ref{eq:R:Rn(btnp1n)}), $R_{n-1}$
given by (\ref{eq:R:Rnm1(btnp1n)}) and $r_n$ given by (\ref{eq:R:rn(bti)}) into (\ref{eq:R:Rec_Rn_Rnm1_rn}).
The equation (\ref{eq:R:D_Eqn_btn_p1n^2_Intro}), quadratic in $\P_1(n)$ follows, completing the proof of
Theorem \ref{thm:R:p1(n)^2Eqn}.
\end{proof}

\subsection{Proof of Theorem \ref{thm:R:3rd_Ord_D_Eqn_btn}}
Equipped with (\ref{eq:R:D_Eqn_btn_p1n^2_Intro}), we solve for $\P_1(n)$ to obtain a third order
difference equation satisfied by  $\bt_n$, and prove Theorem \ref{thm:R:3rd_Ord_D_Eqn_btn}.

\begin{proof}
Solving for $\P_1(n)$ we find,
\bea\label{eq:R:p1n(btngn)}
2k^2\P_1(n)&=&-2k^2\Big(\al+\bt+n-\frac{1}{2}\Big)\bt_n+\bt+\al k^2\pm g_n,
\eea
with $g_n$ given by
\bea\label{eq:R:g(n)Def}
g_n^2&=&4k^4\Big(\al+\bt+n+\frac{1}{2}\Big)\Big(\al+\bt+n-\frac{3}{2}\Big)\bt_n\bt_{n-1}
+4k^4\Big(\al+\bt+n+\frac{3}{2}\Big)\Big(\al+\bt+n-\frac{1}{2}\Big)\bt_{n+1}\bt_n\nonumber\\
&&+8k^4\Big(\al+\bt+n+\frac{1}{2}\Big)\Big(\al+\bt+n-\frac{1}{2}\Big)\bt_n^2
-4k^2(k^2+1)\Big(\al+\bt+n+\frac{1}{2}\Big)\Big(\al+\bt+n-\frac{1}{2}\Big)\bt_n\nonumber\\
&&+(\al k^2+\bt)^2+k^2n(n+2\al+2\bt).
\eea
Making the shift $n\to n+1$ in \eqref{eq:R:p1n(btngn)} leads to
\bea\label{eq:R:p1(n+1)(btngn)}
2k^2\P_1(n+1)&=&-2k^2\Big(\al+\bt+n+\frac{1}{2}\Big)\bt_{n+1}+\bt+\al k^2\pm g_{n+1}.
\eea
If we subtract (\ref{eq:R:p1(n+1)(btngn)}) from (\ref{eq:R:p1n(btngn)}), and note that $\bt_n=\P_1(n)-\P_1(n+1),$
a little simplification gives,
\bea\label{eq:R:3ODbtnv2}
0&=&2k^2\Big(\al+\bt+n+\frac{1}{2}\Big)\big(\bt_{n+1}-\bt_n\big)\mp g_{n+1}\pm g_n.
\eea
This is a third order difference equation satisfied by $\bt_{n+2}$, $\bt_{n+1}$, $\bt_n$ and $\bt_{n-1}$.

All that remains to be done is to choose the correct
sign in (\ref{eq:R:3ODbtnv2}). Taking the sum of (\ref{eq:R:3ODbtnv2}) from $j=0$ to $j=n-1$, we
would like to recover (\ref{eq:R:p1n(btngn)}). With the aid of the identity
\bea
\sum_{j=0}^{n-1}j\Big(\bt_{j+1}-\bt_j\Big)&=&n\bt_n-\sum_{j=0}^{n-1}\bt_{j+1},
\eea
the 
sum of (\ref{eq:R:3ODbtnv2}) from $j=0$ to $j=n-1$ simplifies to
\bea\label{eq:R:TelSum(btnp1ngn)}
0&=&2k^2\Big(\al+\bt+n-\frac{1}{2}\Big)\bt_n+2k^2\P_1(n)\mp g_n\pm g_0.
\eea
From (\ref{eq:R:g(n)Def}), we see that $g_0=\al k^2+\bt$. We take the ``lower" signs to make (\ref{eq:R:TelSum(btnp1ngn)})
compatible with (\ref{eq:R:p1n(btngn)}).

Hence $\bt_n$ satisfies the following third order difference equation:
\bea\label{eq:R:3ODbtn}
0&=&2k^2\Big(\al+\bt+n+\frac{1}{2}\Big)\big(\bt_{n+1}-\bt_n\big)+g_{n+1}-g_n.
\eea

Finally, 
removing the square roots
from (\ref{eq:R:3ODbtn}) we obtain
\bea\label{eq:R:3rdOrdDiffEqnBtnFull}
16k^4\Big(\al+\bt+n+\frac{1}{2}\Big)^2\big(\bt_{n+1}-\bt_n\big)^2g_n^2&=&
\Bigg[g_{n+1}^2-g_n^2-4k^4\Big(\al+\bt+n+\frac{1}{2}\Big)^2(\bt_{n+1}-\bt_n)^2\Bigg]^2, \nonumber\\
\eea
thus completing the proof
of Theorem \ref{thm:R:3rd_Ord_D_Eqn_btn}.
\end{proof}
\subsection{Proof of Theorem \ref{thm:R:Rees_General}}

We state a very simple lemma, before the proof of Theorem \ref{thm:R:Rees_General}.
\begin{lemma}
$\bt_{n-1}$, $\P_1(n)$ and $\P_1(n-1)$ satisfy the identity:
\bea\label{eq:R:p1n2Ident}
[\P_1(n-1)]^2-[\P_1(n)]^2=\bt_{n-1}^2+2\bt_{n-1}\P_1(n).
\eea
\end{lemma}
\begin{proof}[Proof of Theorem \ref{thm:R:Rees_General}]
We isolate $[\P_1(n)]^2$ in (\ref{eq:R:D_Eqn_btn_p1n^2_Intro}) to obtain
\begin{small}
\bea\label{eq:R:p1n^2}
k^2[\P_1(n)]^2&=&-\Bigg[2k^2\Big(\al+\bt+n-\frac{1}{2}\Big)\bt_n-\al k^2-\bt\Bigg]\P_1(n)+k^2\Big(\al+\bt+n+\frac{3}{2}\Big)\Big(\al+\bt+n-\frac{1}{2}\Big)\bt_n^2\nonumber\\
&&+\Bigg[k^2\Big(\al+\bt+n+\frac{3}{2}\Big)\Big(\al+\bt+n-\frac{1}{2}\Big)\bt_{n+1}-\bigg(\Big(\bt+n+\frac{1}{2}\Big)k^2+\Big(\al+n+\frac{1}{2}\Big)\bigg)\Big(\al+\bt+n-\frac{1}{2}\Big)\nonumber\\
&&\qquad+k^2\Big(\al+\bt+n+\frac{1}{2}\Big)\Big(\al+\bt+n-\frac{3}{2}\Big)\bt_{n-1}\Bigg]\bt_n+\frac{n}{2}\left(\frac{n}{2}+\al+\bt\right).
\eea
\end{small}\noindent
We 
now replace [$\P_1(n)]^2$ in (\ref{eq:R:p1n2Ident}) by the right side of (\ref{eq:R:p1n^2}).
In order to do the same for $[\P_1(n-1)]^2$, we use (\ref{eq:R:p1n^2}) again, 
but with  $n$ replaced by $n-1$. This equation now contains a linear term in $\P_1(n-1)$, which
is $\P_1(n)+\bt_{n-1}.$ 
This gives
\bea\label{eq:R:3rdOrdDiffEqnBtnp1n}
0&=&
4k^2(\bt_n-\bt_{n-1})\P_1(n)-2k^2\Big(\al+\bt+n+\frac{3}{2}\Big)\bt_n\big(\bt_n+\bt_{n+1}\big)+
2\bigg[\Big(\al+n+\frac{1}{2}\Big)+\Big(\bt+n+\frac{1}{2}\Big)k^2\bigg]\bt_n
\nonumber\\
&&+2k^2\Big(\al+\bt+n-\frac{5}{2}\Big)\bt_{n-1}(\bt_{n-1}+\bt_{n-2})-2\bigg[\Big(\al+n-\frac{3}{2}\Big)+
\Big(\bt+n-\frac{3}{2}\Big)k^2\bigg]\bt_{n-1}-1.
\eea
In order to re-write the above equation consistent with Rees'  equation
we replace  $\bt_{n-1}\P_1(n)$ in (\ref{eq:R:3rdOrdDiffEqnBtnp1n}) using the easy identity,
$$
\bt_{n-1}\P_1(n)=\bt_{n-1}\big(\P_1(n-2)-\bt_{n-1}-\bt_{n-2}\big),
$$
to obtain
\bea
0&=&\bt_{n-1}\Bigg[2k^2\Big(\al+\bt+n-\frac{1}{2}\Big)(\bt_{n-1}+\bt_{n-2})-2
\bigg[\Big(\al+n-\frac{3}{2}\Big)+\Big(\bt+n-\frac{3}{2}\Big)k^2\bigg]-4k^2\P_1(n-2)\Bigg]\nonumber\\
&&-\bt_{n}\Bigg[2k^2\Big(\al+\bt+n+\frac{3}{2}\Big)(\bt_{n}+\bt_{n+1})-2
\bigg[\Big(\al+n+\frac{1}{2}\Big)+\Big(\bt+n+\frac{1}{2}\Big)k^2\bigg]-4k^2\P_1(n)\Bigg]-1.\nonumber\\
\eea
This is (\ref{eq:R:Gen_Rees_D_Eqn}) with
$C_n$ given by (\ref{eq:R:Cn}).
\end{proof}
\begin{remark}
To reiterate, equation (\ref{eq:R:Gen_Rees_D_Eqn}), or
its equivalent form (\ref{eq:R:3rdOrdDiffEqnBtnp1n}) are essentially generalizations of Rees' equation,
valid for $\al>-1$ and $\bt\in\mathbb{R}$. 
Had we eliminated
$\P_1(n)$ from (\ref{eq:R:3rdOrdDiffEqnBtnp1n}), we would have obtained the following fourth order difference equation satisfied by $\bt_n$:
\begin{small}
\bea\label{eq:R:Rees4ODEqn}
2k^2\big(\bt_{n+1}-\bt_n\big)\bt_n\big(\bt_n-\bt_{n-1}\big)&=&
k^2\Big(\al+\bt+n+\frac{3}{2}\Big)\big(\bt_{n+1}^2-\bt_n^2\big)\bt_n+k^2\Big(\al+\bt+n-\frac{3}{2}\Big)\bt_n\big(\bt_n^2-\bt_{n-1}^2\big)\nonumber\\
&&+(k^2+1)\bt_n\big(\bt_{n+1}+\bt_n+\bt_{n-1}\big)
-3(k^2+1)\bt_{n+1}\bt_{n-1}+\frac{\bt_{n+1}}{2}-\bt_n+\frac{\bt_{n-1}}{2}\nonumber
\\&&
-k^2\Big(\al+\bt+n-\frac{5}{2}\Big)\big(\bt_{n+1}-\bt_n\big)\bt_{n-1}\big(\bt_{n-1}+\bt_{n-2}\big)\nonumber\\
&&-k^2\Big(\al+\bt+n+\frac{5}{2}\Big)\big(\bt_{n+2}+\bt_{n+1}\big)\bt_{n+1}\big(\bt_{n}-\bt_{n-1}\big).
\eea
\end{small}\noindent

Our difference equation (\ref{eq:R:3rd_Ord_D_Eqn_btn_Intro}) is one order lower 
compared with  
(\ref{eq:R:Rees4ODEqn}). The third
 order difference equation arises mainly due to (\ref{eq:R:D_Eqn_btn_p1n^2_Intro}), an equation that is
 not contained in the formalism of Rees.

The important idea is that if we instead start from 
\eqref{eq:R:3rdOrdDiffEqnBtnp1n}, we can
obtain our third order difference equation (\ref{eq:R:3rd_Ord_D_Eqn_btn_Intro})
{\em if we combine (\ref{eq:R:3rdOrdDiffEqnBtnp1n}) with (\ref{eq:R:D_Eqn_btn_p1n^2_Intro})}.
Currently, this appears to be the only way 
of eliminating $\P_1(n)$ from \eqref{eq:R:3rdOrdDiffEqnBtnp1n} without increasing the order of the difference equation from third to fourth.
\end{remark}

\section{Second order difference equations for  \texorpdfstring{$\bt_n$}{recurrence coefficient} and  \texorpdfstring{$\P_1(n)$}{sub-leading term}}\label{Sec:R:DiffEqn_btn_p1n}

In this section we prove Theorems \ref{thm:R:2nd_Ord_D_Eqn} and \ref{thm:R:2D_Eqn_p1n}.
These are second order difference equations satisfied by $\bt_n$ and $\P_1(n)$.

We prove Theorem \ref{thm:R:2nd_Ord_D_Eqn} by first establishing 3 algebraic equations
satisfied by $r_n$, $\bt_n$ and $\P_1(n)$. The `coefficients' of these, depending on $\bt_{n+1}$, $\bt_{n-1}$, $n$,
$k^2$, $\al$ and $\bt$, are treated as constants. We then use MAPLE's elimination algorithm to
eliminate $r_n$ and $\P_1(n)$ and we are left with the `variable' $\bt_n$,
expressed in terms of $\bt_{n+1}$, $\bt_{n-1}$, $n$, $k^2$, $\al$ and $\bt$. This is the second-order
difference equation satisfied by $\bt_n,$ mentioned above.

\begin{lemma}\label{lem:R:2nd_Ord_Sys_(rnbtnp1n)}
The terms $r_n$, $\bt_n$ and $\P_1(n)$ satisfy the following system of 3 equations:
\bea
\label{eq:R:2nd_Ord_Sys_1}(k^2-1)r_n&=&\frac{n}{2}-k^2\left(\al+\bt+n+\frac{1}{2}\right)\beta_n+k^2\textsf{p}_1(n),
\eea
\begin{small}
\bea\label{eq:R:2nd_Ord_Sys_2}
r_n(r_n-\al)&=&
\bt_n
\Bigg[\al-
\frac{1}{k^2-1}\left(n+\frac{1}{2}-k^2\left(\al+\bt+n+\frac{3}{2}\right)(\beta_n+\bt_{n+1})+2k^2\textsf{p}_1(n)\right)
\Bigg]\nonumber\\
&&\;\times\Bigg[\al-
\frac{1}{k^2-1}\left(n-\frac{1}{2}-k^2\left(\al+\bt+n+\frac{1}{2}\right)\beta_n-k^2\left(\al+\bt+n-\frac{3}{2}\right)\bt_{n-1}+2k^2\P_1(n)\right)
\Bigg],
\eea
\bea\label{eq:R:2nd_Ord_Sys_3}
(k^2-1)r_n^2-(n+\bt+\al k^2)r_n-\frac{n}{2}\left(\frac{n}{2}+\bt\right)
+k^2\bt_n
\Bigg[\al+\left(\al+\bt+n+\frac{1}{2}\right)\left(-\al+\bt+n-\frac{1}{2}\right)
\nonumber\\
+\frac{1}{k^2-1}\Bigg\{2(\al+\bt+n)\Big(n+2k^2\P_1(n)\Big)-\frac{1}{2}-k^2\Bigg(2\left(\al+\bt+n+\frac{1}{2}\right)^2-1\Bigg)\bt_n
\nonumber\\
-k^2\left(\al+\bt+n+\frac{3}{2}\right)\left(\al+\bt+n-\frac{1}{2}\right)\bt_{n+1}-k^2
\left(\al+\bt+n+\frac{1}{2}\right)\left(\al+\bt+n-\frac{3}{2}\right)\bt_{n-1}\Bigg\}\Bigg]
&=&0.
\eea
\end{small}
\end{lemma}
\begin{proof}
Equation (\ref{eq:R:2nd_Ord_Sys_1}) is a restatement of (\ref{eq:R:p1n(rnbtn)}).
To obtain (\ref{eq:R:2nd_Ord_Sys_2}), we eliminate $R_n$ and $R_{n-1}$ in (\ref{sys:R:MCS2p1.1})
using (\ref{eq:R:Rn(btnp1n)}) and (\ref{eq:R:Rnm1(btnp1n)}) respectively.

Equation (\ref{eq:R:2nd_Ord_Sys_3}) may be obtained in two steps. First, we restate (\ref{eq:R:cnRn-1_cn-1Rn(btnrn)}) as,
\bea
(k^2-1)r_n^2-r_n(n+\bt+\al k^2)-\frac{n}{2}\left(\frac{n}{2}+\bt\right)
&=&k^2\bt_n\Bigg[\left(\al+\bt+n+\frac{1}{2}\right)R_{n-1}+\left(\al+\bt+n-\frac{1}{2}\right)R_n
\nonumber\\&&
-\left(\al+\bt+n+\frac{1}{2}\right)\left(\al+\bt+n-\frac{1}{2}\right)\Bigg].
\eea
Second, we eliminate $R_n$ and $R_{n-1}$ from the above with (\ref{eq:R:Rn(btnp1n)}) and (\ref{eq:R:Rnm1(btnp1n)})
to obtain (\ref{eq:R:2nd_Ord_Sys_3}).
\end{proof}
\subsection{Proof of Theorem \ref{thm:R:2nd_Ord_D_Eqn}}
\begin{proof}
Equations (\ref{eq:R:2nd_Ord_Sys_1})--(\ref{eq:R:2nd_Ord_Sys_3}) may be regarded as a system of non-linear
algebraic equations satisfied by $r_n$, $\bt_n$ and $\P_1(n)$, where we
treat $\bt_{n+1}$, $\bt_{n-1}$, $n$, $k^2$, $\al$ and $\bt$ as constants.

MAPLE's elimination algorithm is then applied to express $\P_1(n)$ and $r_n$ in terms of $\bt_n$ as
\begin{scriptsize}
\bea\label{eq:R:p1n(btn)}
4k^2K_n\P_1(n)&=&
2\Bigg[2{k}^2 \left(\alpha+\beta+n+\frac{3}{2}\right) \beta_{{n+1}}+10{k}^2
\left(\alpha+\beta+n+{\frac 7 {10}}\right) \beta_{{n}}
-(2\bt+2n+1){k}^2 -(2\al+2n+1)\Bigg]
 \left(\alpha+\beta+n-\frac{3}{2}\right) k^2\beta_{{n}}\beta_{{n-1}}
\nonumber\\&&
+2 \Bigg[10\left(\alpha+\beta+n-\frac{3}{10}\right) {k}^2 \beta_{{n}}-(2\bt+2 n-1) {k}^2 -(2\al+2n-1)\Bigg]
 \left(\alpha+\beta+n+\frac{3}{2}\right) {k}^2 \beta_{{n}}\beta_{{n+1}}
\nonumber\\&&
+20 \left(\alpha+\beta+n+\frac{3}{2}\right) \left(\alpha+\beta+n-\frac{3}{10}\right) {k}^4 {\beta_{{n}}}^3
-24 \bigg[k^2(\al+\bt+n)\left(\bt+n+\frac{1}{3}\right)+(\al+\bt+n)\left(\al+n+\frac{1}{3}\right)-\frac{1}{4}(k^2+1)\bigg] k^2{\beta_{{n}}}^2
\nonumber\\&&
+\bigg[(2\bt+2n+1)(2\bt+2n-1){k}^4 +\Big(12 {n}^2+2n-1+(16n+2)(\al+\bt)+8\al\bt\Big ) {k}^2 +(2\al+2n+1)(2\al+2n-1)\bigg] \beta_{{n}}
\nonumber\\&&
-2 n\bigg[\left(\bt+\frac{n}{2}\right) {k}^2 +\alpha+\frac{n}{2}\bigg],
\eea
\end{scriptsize}\noindent
and
\begin{scriptsize}
\bea\label{eq:R:rn(btn)}
4(k^2-1)K_nr_n&=&
2 \Bigg[2{k}^2 \left(\alpha+\beta+n+\frac{3}{2}\right) \beta_{{n+1}}+
6{k}^2 \left(\alpha+\beta+n+\frac{5}{6}\right) \beta_{{n}}-(2\bt+2n+1) {k}^2 -(2\al+1)\Bigg]\left(\alpha+\beta+n-\frac{3}{2}\right)k^2\bt_n\beta_{{n-1}}
\nonumber\\&&
+2 \Bigg[6{k}^2 \left(\alpha+\beta+n-\frac{5}{6}\right) \beta_{{n}}-(2\bt+2n-1)
 {k}^2 -(2 \alpha-1)\Bigg] \left(\alpha+\beta+n+\frac{3}{2}\right) k^2\beta_{{n}}\beta_{{n+1}}
-\Big(28(\al+\bt+n)^2+9\Big){k}^4 {\beta_{{n}}}^3
\nonumber\\&&
+8\Bigg[(k^2+1)(\al+\bt+n)\left(2\al-\bt-n\right)+\frac{3}{4}(k^2+1)-3(\al+\bt+n)(\al-n-\bt)\Bigg] k^2{\beta_{{n}}}^2
\nonumber\\&&
+\bigg[(2\bt+2n+1)(2\bt+2n-1){k}^4 -(4 {\alpha}^2 +4 {\beta}^2+1 ) {k}^2 +4(\al^2 -{n}^2) -8 n\beta-1\bigg]
\beta_{{n}}-2n ({k}^2 -1) \left(\bt+\frac{n}{2}\right),
\eea
\end{scriptsize}\noindent
respectively. Here  $K_n$ is defined by
\bea
K_n&:=&K(\bt_{n+1},\bt_n,\bt_{n-1},k^2,n,\al,\bt),
\nonumber\\
&=&
\Big(1+12\,{k}^{2}{\beta_{{n}}}^{2}+ 2\Big(  \left( \beta_{{n+1}}+
\beta_{{n-1}}-2 \right) {k}^{2}-2 \Big) \beta_{{n}}\Big)(\al+\bt+n)
+3k^2\beta_{{n}} \left( \beta_{{n+1}}-\beta_{{n-1}} \right).
\eea
We now substitute $\P_1(n)$ and $r_n$ given by (\ref{eq:R:p1n(btn)}) and (\ref{eq:R:rn(btn)}) respectively,
into (\ref{eq:R:2nd_Ord_Sys_1}) and obtain Theorem~\ref{thm:R:2nd_Ord_D_Eqn} after simplifications.
\end{proof}
%
%
%
\subsection{Proof of Theorem \ref{thm:R:2D_Eqn_p1n}}

Substituting
\bea
\label{eq:R:bt_n+1(p1n)}\bt_{n+1}&=&\P_1(n+1)-\P_1(n+2),\\
\label{eq:R:bt_n(p1n)}\bt_{n}&=&\P_1(n)-\P_1(n+1),\\
\label{eq:R:bt_n-1(p1n)}\bt_{n-1}&=&\P_1(n-1)-\P_1(n),
\eea
into the second order difference equation for $\bt_n$, (\ref{eq:R:2nd_Ord_Diff_Eqn}), we obtain a third order difference equation
satisfied by $\P_1(n)$.
However, using the same method used in the proof of Theorem \ref{thm:R:2nd_Ord_D_Eqn}, a second order
difference equation satisfied by $\P_1(n)$ may be obtained.

We give an outline of the proof.

The crucial idea is to use the system of algebraic equations in Lemma~\ref{lem:R:2nd_Ord_Sys_(rnbtnp1n)},
(\ref{eq:R:2nd_Ord_Sys_1})--(\ref{eq:R:2nd_Ord_Sys_3}). We substitute $\bt_{n+1}$, $\bt_n$ and $\bt_{n-1}$ into this system
 using (\ref{eq:R:bt_n+1(p1n)})--(\ref{eq:R:bt_n-1(p1n)}) to obtain a system of equations satisfied by the variables
 $r_n$, $\P_1(n+2)$, $\P_1(n+1)$, $\P_1(n)$ and $\P_1(n-1)$. By using MAPLE's elimination algorithm to eliminate
 $r_n$ and $\P_1(n+2)$, from the 3 equations, we are left with a second order
  difference equation satisfied by $\P_1(n+1)$, $\P_1(n)$ and $\P_1(n-1)$, thus completing the proof of Theorem~\ref{thm:R:2D_Eqn_p1n}.

%
%

\subsection{Proof of Theorem \ref{thm:R:Diff_Eqn_Pn}}

The ladder operator approach leads directly to a second order linear ode satisfied by the orthogonal polynomials. For the given weight, the equation can be described by exactly the same auxiliary variables that were used to find the previous difference equations.

In this case, eliminating $P_{n-1}(x)$ from (\ref{Defn:LadderOp}), we have the ode
\bea\label{eq:R:Diff_Eqn_Pn}
P_n^{\;\prime\prime}(x)
-\left(
\textsf{v}^\prime(x)
+\frac{A_n^\prime(x)}{A_n(x)}\right)P_n^{\;\prime}(x)
+\left(B_n^\prime(x)-B_n(x)\frac{A_n^\prime(x)}{A_n(x)}+\sum\limits_{j=0}^{n-1}A_j(x)\right)P_n(x)&=&0.
\eea
From (\ref{eq:R:vp(x)}), (\ref{eq:R:An}) and (\ref{eq:R:Bn}), we may rewrite $\textsf{v}^\prime(x)$, $A_n(x)$ and $B_n(x)$ as the following rational functions in $x$:
\bea
\label{eq:R:vprime_Pn_Proof}
\textsf{v}^{\prime}(x)=-\Big(\log w(x)\Big)^\prime&=&-\frac{2x\Big[k^2(\al+\bt)x^2-\al-\bt k^2\Big]}{(x^2-1)(k^2x^2-1)},\\
\label{eq:R:An_Pn_Proof}
A_n(x)&=&-\frac{2\Big[R_n(k^2x^2-1)+kR_n^\ast(x^2-1)\Big]}{(x^2-1)(k^2x^2-1)},\\
\label{eq:R:Bn_Pn_Proof}
B_n(x)&=&\frac{2x\Big[r_n(k^2x^2-1)+k^2r_n^\ast(x^2-1)\Big]}{(x^2-1)(k^2x^2-1)},
\eea
where
$R_n$, $r_{n}$, $R^{\ast}_{n}$, and $r^{\ast}_{n}$ are the auxiliary variables, defined by (\ref{eq:R:AuxVar}).

We also know that  $R_n^\ast$ can be expressed in terms of $R_n$ as (\ref{eq:R:RnRnastv1}),
\bea\label{eq:R:RnRnastv2}
\frac{R_n^\ast}{k}&=&\al+\bt+n+\frac{1}{2} -R_n,
\eea
while $r_n^\ast$ can be expressed in terms of $r_n$ as (\ref{eq:R:rnrnast}),
\bea\label{eq:R:rnrnast_2}
r_n^\ast&=&-r_n-\frac{n}{2}.
\eea

Now, we can express $R_n$ in terms of $\bt_n$, $\bt_{n+1}$ and $\P_1(n)$ as (\ref{eq:R:Rn(btnp1n)}),
\bea\label{eq:R:Rn(btnp1n)_2}
R_n&=&
\al-
\frac{1}{k^2-1}\left[n+\frac{1}{2}-k^2\Big(\al+\bt+n+\frac{3}{2}\Big)(\beta_n+\bt_{n+1})+2k^2\textsf{p}_1(n)\right],
\eea
and $r_n$ may be expressed in terms of $\bt_n$ and $\P_1(n)$ as (\ref{eq:R:p1n(rnbtn)}),
\bea
\label{eq:R:rn(btnp1n)}
r_n&=&\frac{1}{k^2-1}\left[\frac{n}{2}-k^2\Big(\al+\bt+n+\frac{1}{2}\Big)\beta_n+k^2\textsf{p}_1(n)\right].
\eea

\subsubsection{Coefficient of $P_n^{\;\prime}(x)$ in (\ref{eq:R:Diff_Eqn_Pn})}
To calculate the coefficient of $P_n^{\;\prime}(x)$ in (\ref{eq:R:Diff_Eqn_Pn}), we calculate the value of $-\textsf{v}^\prime(x)-A_n^\prime(x)/A_n(x)$.

To calculate $A_n^\prime(x)/A_n(x)$, we notice that $\frac{d}{dx}\log\left(\frac{p(x)}{q(x)}\right)=\frac{p^\prime(x)}{p(x)}-\frac{q^\prime(x)}{q(x)}$, and hence
\begin{small}
\bea
\frac{A_n^\prime(x)}{A_n(x)}&=&
        \frac{2x\Big[k^2R_n+kR_n^\ast\Big]}{\Big[R_n(k^2x^2-1)+kR_n^\ast(x^2-1)\Big]}
        -\frac{2x\Big[2k^2x^2-1-k^2\Big]}{(x^2-1)(k^2x^2-1)}.
\eea

To simplify the first term (squared term in $x$ in the denominator), we proceed with the following steps:
\begin{enumerate}
\item
We eliminate $R_n^\ast$ in favor of $R_n$ in both the numerator and denominator using (\ref{eq:R:RnRnastv2}).
\item
We eliminate $R_n$ in favor of $\bt_n$, $\bt_{n+1}$ and $\P_1(n)$ using (\ref{eq:R:Rn(btnp1n)_2}).
\item
We simplify and identify $M_n(x)$ and $Y(x)$.
\end{enumerate}
Hence
\bea
\frac{A_n^\prime(x)}{A_n(x)}&=&
\frac{2xk^2\Big[\al+\bt+n+\frac{1}{2}\Big]}{\Big[(x^2-1)k^2(\al+\bt+n+\frac{1}{2})+(k^2-1)R_n\Big]}
-\frac{2x\Big[2k^2x^2-1-k^2\Big]}{(x^2-1)(k^2x^2-1)},\\
&=&
\frac{2xk^2\Big[\al+\bt+n+\frac{1}{2}\Big]}{\Big[(x^2-1)k^2(\al+\bt+n+\frac{1}{2})+(k^2-1)\al-(n+\frac{1}{2})+k^2(\al+\bt+n+\frac{3}{2})(\bt_n+\bt_{n+1})-2k^2\P_1(n)\Big]}
\nonumber\\&&
-\frac{2x\Big[2k^2x^2-1-k^2\Big]}{(x^2-1)(k^2x^2-1)},\\
&=&
\frac{M_n^{\prime}(x)}{M_n(x)}
-\frac{2x\Big[2k^2x^2-1-k^2\Big]}{(x^2-1)(k^2x^2-1)},\\
\label{eq:R:Anp/An(MnYn)}
&=&
\frac{M_n^{\prime}(x)}{M_n(x)}-\frac{Y^\prime(x)}{Y(x)},
\eea
\end{small}
where
\bea
M_n(x)&:=&-2\left(\al+\bt+n+\frac{1}{2}\right)k^2x^2-C_n,\\
C_n&:=&2\left(\al+\bt+n+\frac{3}{2}\right)k^2(\bt_n+\bt_{n+1})-2\left[\left(\bt+n+\frac{1}{2}\right)k^2+\al+n+\frac{1}{2}\right]-4k^2\P_1(n),\qquad\\
Y(x)&:=&(1-x^2)(1-k^2x^2).
\eea
Hence combining (\ref{eq:R:vprime_Pn_Proof}) and (\ref{eq:R:Anp/An(MnYn)}), the coefficient of $P_n^\prime(x)$ in (\ref{eq:R:Diff_Eqn_Pn}) is given by
\begin{small}
\bea
-\textsf{v}^{\prime}(x)-\frac{A_n^\prime(x)}{A_n(x)}&=&
-\frac{M_n^\prime(x)}{M_n(x)}
+\frac{2x\Big[2k^2x^2-1-k^2\Big]}{(x^2-1)(k^2x^2-1)}
+\frac{2x\Big[k^2x^2(\al+\bt)-\al-\bt k^2\Big]}{(x^2-1)(k^2x^2-1)},\\
&=&
-\frac{M_n^\prime(x)}{M_n(x)}
+\frac{2x\Big[k^2x^2(\al+\bt+2)-1-\al-k^2(1+\bt)\Big]}{(x^2-1)(k^2x^2-1)},\\
&=&
-\frac{M_n^\prime(x)}{M_n(x)}+\frac{1}{2}\frac{X^\prime(x)}{X(x)},
\eea
\end{small}
where
\bea
X(x)&:=&(1-x^2)^{2\al+2}(1-k^2x^2)^{2\bt+2}.
\eea
\subsubsection{Coefficient of $P_n(x)$ in (\ref{eq:R:Diff_Eqn_Pn})}
To calculate the coefficient of $P_n(x)$ in (\ref{eq:R:Diff_Eqn_Pn}), we  calculate the value of $B_n^\prime(x)-B_n(x)\frac{A_n^\prime(x)}{A_n(x)}+\sum\limits_{j=0}^{n-1}A_j(x)$.

We first eliminate $R_n^\ast$ in favor of $R_n$ in (\ref{eq:R:An_Pn_Proof}) using (\ref{eq:R:RnRnastv2}).
\begin{small}
\bea
A_n(x)&=&
-\frac{2\Big[x^2(k^2R_n+kR_n^\ast)-R_n-kR_n^\ast\Big]}{(x^2-1)(k^2x^2-1)},\\
\label{eq:R:Pn_Proof_SumAn_1}
&=&
-\frac{2\Big[k^2(x^2-1)(\al+\bt+n+\frac{1}{2})+(k^2-1)R_n\Big]}{(x^2-1)(k^2x^2-1)}.
\eea
\end{small}\noindent
We then take a telescopic sum of (\ref{eq:R:Pn_Proof_SumAn_1}) from $j=0$ to $j=n-1$, where we use the following result obtained from Theorem~\ref{lem:R:SumR(btnrn)}:
\begin{small}
\bea\label{eq:R:SumR(btnrn)_2}
\frac{1}{2}\sum\limits_{j=0}^{n-1}\!R_j&=&
        -r_n(\alpha+\beta+n)+\frac{n}{2}\left(\frac{n}{2}+\beta+\alpha k^2\right)\frac{1}{k^2-1}
        -\frac{k^2}{k^2-1}\left(\al+\bt+n+\frac{1}{2}\right)\left(\al+\bt+n-\frac{1}{2}\right)\bt_n.
\eea
\end{small}
We proceed with the following steps:
\begin{enumerate}
\item
We replace $\sum\limits_{j=0}^{n-1}R_j$ in the telescopic sum of (\ref{eq:R:Pn_Proof_SumAn_1}) using (\ref{eq:R:SumR(btnrn)_2}).
\item
We replace $r_n$ by $\bt_n$ and $\P_1(n)$ using (\ref{eq:R:rn(btnp1n)}).
\end{enumerate}
Hence
\begin{small}
\bea
\sum\limits_{j=0}^{n-1}A_j(x)&=&
-\frac{2k^2(x^2-1)n\left(\al+\bt+\frac{n}{2}\right)}{(x^2-1)(k^2x^2-1)}
\nonumber\\&&
+\frac{4(k^2-1)r_n(\al+\bt+n)-2n\left(\frac{n}{2}+\bt+\al k^2\right)+4k^2\left(\al+\bt+n+\frac{1}{2}\right)\left(\al+\bt+n-\frac{1}{2}\right)\bt_n}{(x^2-1)(k^2x^2-1)},\\
&=&
\frac{-2k^2(x^2-1)n\left(\al+\bt+\frac{n}{2}\right)+4\Big[\frac{n}{2}-k^2\left(\al+\bt+n+\frac{1}{2}\right)\bt_n+k^2\P_1(n)\Big](\al+\bt+n)}{(x^2-1)(k^2x^2-1)}
\nonumber\\&&
+\frac{-2n\left(\frac{n}{2}+\bt+\al k^2\right)+4k^2\left(\al+\bt+n+\frac{1}{2}\right)\left(\al+\bt+n-\frac{1}{2}\right)\bt_n}{(x^2-1)(k^2x^2-1)},\\
\label{eq:R:Pn_Proof_SumAn_2}
&=&
\frac{-k^2x^2n(n+2\al+2\bt)+nk^2(n+2\bt)+n(n+2\al)+4k^2(\al+\bt+n)\P_1(n)-2k^2\left(\al+\bt+n+\frac{1}{2}\right)\bt_n}{(x^2-1)(k^2x^2-1)}.\nonumber\\
\eea
\end{small}
To compute $B_n(x)$ we have
from (\ref{eq:R:Bn_Pn_Proof}),  that $B_n(x)$ is given by
\bea
B_n(x)&=&\frac{2x\Big[k^2x^2(r_n+r_n^\ast)-r_n-k^2r_n^\ast\Big]}{(x^2-1)(k^2x^2-1)}.\nonumber
\eea
To simplify $B_n(x)$, we proceed in the following way:
\begin{enumerate}
\item
We eliminate $r_n^\ast$ in favor of $r_n$ in $(20)$ using (\ref{eq:R:rnrnast_2}).
\item
We then eliminate $r_n$ in favor of $\bt_n$ and $\P_1(n)$ using (\ref{eq:R:rn(btnp1n)}).
\item
We simplify and identify $L_n(x)$.
\end{enumerate}
Hence
\bea
B_n(x)&=&
\frac{2x\Big[-\frac{n}{2}k^2x^2+(k^2-1)r_n+\frac{n}{2}k^2\Big]}{(x^2-1)(k^2x^2-1)},\\
&=&
\frac{x\Big[-nk^2x^2+n(1+k^2)-2k^2(\al+\bt+n+\frac{1}{2})\bt_n+2k^2\P_1(n)\Big]}{(x^2-1)(k^2x^2-1)},\\
\label{eq:R:Pn_Proof_Bn_1}
&=&
-\frac{L_n(x)}{Y(x)},
\eea
where
\bea
L_n(x)&:=&x\Big[nk^2x^2-n(1+k^2)+2k^2\left(\al+\bt+n+\frac{1}{2}\right)\bt_n-2k^2\P_1(n)\Big].
\eea

Hence, combining (\ref{eq:R:Anp/An(MnYn)}), (\ref{eq:R:Pn_Proof_SumAn_2}) and (\ref{eq:R:Pn_Proof_Bn_1}), we have
\begin{small}
\bea
B_n^\prime(x)-B_n(x)\frac{A_n^\prime(x)}{A_n(x)}+\sum\limits_{j=0}^{n-1}A_j(x)&=&
-\frac{L_n(x)/x+2nk^2x^2}{Y(x)}+\frac{L_n(x)M_n^\prime(x)}{Y(x)M_n(x)}
\nonumber\\&&
+\frac{1}{Y(x)}\bigg[-k^2x^2n(n+2\al+2\bt)+nk^2(n+2\bt)+n(n+2\al)
\nonumber\\&&\qquad\qquad
+4k^2(\al+\bt+n)\P_1(n)-2k^2\Big(\al+\bt+n+\frac{1}{2}\Big)\bt_n\bigg],\\
&=&
\frac{U_n(x)}{Y(x)}+\frac{L_n(x)M_n^\prime(x)}{Y(x)M_n(x)},
\eea
\end{small}
where
\bea
U_n(x)&:=&-k^2x^2n(n+2\al+2\bt+3) + 2k^2(2n+2\al+2\bt+1)\P_1(n+1)+nk^2(n+2\bt+1)
\nonumber\\&&
+n(n+2\al+1).
\eea
Hence, we have completed the derivation to equation (\ref{eq:R:DiffEqnPn_Gen}), completing the proof of Theorem~\ref{thm:R:Diff_Eqn_Pn}.

\section{Special solutions of the second and third order difference equation for \texorpdfstring{$\bt_n$}{recurrence coefficient}}\label{Sec:R:Special_Cases}

A number of explicit solutions can be obtained through the specialization of $\al,\;\bt,$ and $k^2.$
In these situations, $\bt_n$ are the essentially the recurrence coefficients of the symmetric
Jacobi polynomials. We present some here to verify the difference equations.

We begin with the third-order difference equation.
\subsection{Reduction to Jacobi Polynomials: Third order difference equation}
\label{sec:R:3_D_Red_Jac}
There are three special choices of $\al$, $\bt$ and $k^2$ where our weight
(\ref{defn:R:w(xk2)})
reduces to the Jacobi weight
\bea\label{def:JacobiWeightSimple}
w^{(\al,\bt)}(x)&=&(1-x)^{\al}(1+x)^{\bt}, \qquad x\in[-1,1], \;\;\; \al,\;\bt>-1.
\eea
Note that the recurrence coefficients for monic Jacobi polynomials are given by \cite{KoekoekLeskySwart2000}
\bea
\bt_n&=&\frac{4n(n+\al)(n+\bt)(n+\al+\bt)}{(2n+\al+\bt)^2(2n+\al+\bt+1)(2n+\al+\bt-1)}.
\eea
Please note that this is a double use of $\beta$ and $\alpha$, but it is in keeping with the standard notation.
The three cases are as follows:
\begin{enumerate}
\item\label{C:R:3DE:JacRed1} $k^2=0$, while $\al$ and $\bt$ remain arbitrary.
\\
For this  case, our weight (\ref{defn:R:w(xk2)}) reduces to
$$w(x,0)=(1-x^2)^\al=w^{(\al,\al)}(x).$$
This is a special case of the Jacobi weight where $\al=\bt.$

Since there is only one parameter in the weight, we should expect
$\bt$ not to appear in the third-order difference equation (\ref{eq:R:3rd_Ord_D_Eqn_btn_Intro}).
It turns out, in this situation, (\ref{eq:R:3rd_Ord_D_Eqn_btn_Intro}) reduces to two equations. One is
\bea\label{eq:R:3_D_Eqn(k=0)}
\beta_{n+1}(2 \alpha + 2n +3) = 1 + \beta_{n}(2 \alpha + 2n -1).
\eea
The other is the same with $\alpha$ replaced with $\alpha + 2\beta.$

Solving the difference equations with the initial
condition 
$\bt_0=0,$ leads to the following solutions: 
\bea
\label{eq:R:btn_Jac_Soln_(al=bt)}\bt_n&=&\frac{n(n+2\al)}{(2n+2\al+1)(2n+2\al-1)},\\
\bt_n&=&\frac{n(n+2\al+4\bt)}{(2n+2\al+4\bt+1)(2n+2\al+4\bt-1)}.
\eea
The $\bt$-independent solution gives the recurrence coefficient associated with the Jacobi weight
$w^{(\al,\al)}(x).$ The second solution, which depends on $\bt,$ is not the recurrence coefficient associated
with that weight. However, it is the correct recurrence coefficient
for the Jacobi weight
$w^{(\al+2\bt,\al+2\bt)}(x)$.
\\
\item\label{C:R:3DE:JacRed2} $\bt=0$, while $\al$ and $k^2$ remain arbitrary.\\
In this case, although our weight (\ref{defn:R:w(xk2)}) reduces to the same
weight as in case \ref{C:R:3DE:JacRed1}, a third-order equation remains.
It is not clear how a solution for $\bt_n$ maybe obtained; however, it can be verified that the known solution (\ref{eq:R:btn_Jac_Soln_(al=bt)}) satisfies the equation.\\
\item\label{C:R:3DE:JacRed3} $k^2=1$, while $\al$ and $\bt$ remains arbitrary.\\
In this case (\ref{defn:R:w(xk2)}) reduces to a special case of the Jacobi weight, i.e.,
$$w^{(\al+\bt,\al+\bt)}(x)=(1-x^2)^{\al+\bt}.$$
Since $\al$ and $\bt$ are arbitrary, this is the same as the weight function for the cases
\ref{C:R:3DE:JacRed1} and \ref{C:R:3DE:JacRed2}, but with $\al+\bt$ replaced by $\al$.
Nevertheless, (\ref{eq:R:3rd_Ord_D_Eqn_btn_Intro}) remains a non-linear third-order difference equation, and while
we are unable to solve it explicitly, we can again readily check that
\bea
\bt_n&=&\frac{n(n+2\al+2\bt)}{(2n+2\al+2\bt+1)(2n+2\al+2\bt-1)},
\eea
is a known solution.
A sub-case of the above where $\al=0$ is also a known solution for the recurrence coefficient of the weight $w^{(\bt,\bt)}(x)$. 

In summary, for cases \ref{C:R:3DE:JacRed1}-\ref{C:R:3DE:JacRed3}, the special values of
$\al$, $\bt$ and $k^2$ lead to what is essentially the same Jacobi weight function (in case \ref{C:R:3DE:JacRed3}, $\al$ is replaced by $\al+\bt$),
$$w^{(\al,\al)}(x)=(1-x^2)^\al.$$
 However, (\ref{eq:R:3rd_Ord_D_Eqn_btn_Intro}) reduces to $3$ distinct difference equations
 for the $3$ cases considered. We check by substitution that all $3$ difference
 equations have the correct solution.
%
\end{enumerate}
\subsection{Reduction to Jacobi Polynomials: Second order difference equation}\label{sec:R:2_D_Red_Jac}
There are three special values of $\al$, $\bt$ and $k^2$ which results in our weight (\ref{defn:R:w(xk2)})
 reducing to the Jacobi weight (\ref{def:JacobiWeightSimple}).
These are:
\begin{enumerate}
\item\label{C:R:2DE:JacRed1} $k^2=0$, while $\al$ and $\bt$ remain arbitrary.\\

Here, the weight (\ref{defn:R:w(xk2)}) reduces to
$$w(x,0)=(1-x^2)^\al=w^{(\al,\al)}(x).$$

Going through with the computations,
 our  second order difference equation (\ref{eq:R:2nd_Ord_Diff_Eqn})
 reduces to the following {\it quadratic} equation:
\begin{small}
\bea\label{eq:R:2_D_Eqn(k=0)}
\left[  \left( \alpha+2\beta+n-\frac{1}{2} \right)  \left( \alpha+2
\beta+n+\frac{1}{2} \right) \beta_{{n}}- \left( \alpha+\beta+\frac{n}{2} \right)
 \left( \beta+\frac{n}{2} \right)  \right] &&
 \nonumber\\
 \times\left[  \left( \alpha+n-\frac{1}{2}
 \right)  \left( \alpha+n+\frac{1}{2} \right) \beta_{{n}}-\frac{n}{2} \left( \alpha+
\frac{n}{2} \right)  \right]
&=&0.
\eea
\end{small}
Solving for $\bt_n$, we find,
\bea
\bt_n&=&\frac{n(n+2\al)}{(2n+2\al+1)(2n+2\al-1)},\\
\bt_n&=&\frac{n(n+2\al+4\bt)}{(2n+2\al+4\bt+1)(2n+2\al+4\bt-1)}.
\eea
The first solution gives the recurrence coefficient $\bt_n$ for the Jacobi weight $w^{(\al,\al)}(x)$, while
the second, $\bt$-dependent solution, 
is the recurrence coefficient for
$w^{(\al+2\bt,\al+2\bt)}(x).$
\begin{remark}
In the situation where $k^2=0$, we have seen, the third-order difference equation (\ref{eq:R:3rd_Ord_D_Eqn_btn_Intro}) reduces to a first-order
difference equation (\ref{eq:R:3_D_Eqn(k=0)}), while our second-order difference equation (\ref{eq:R:2nd_Ord_Diff_Eqn}) reduces
to an {\em algebraic equation} satisfied by $\bt_n$, (\ref{eq:R:2_D_Eqn(k=0)}). Nonetheless, (\ref{eq:R:3_D_Eqn(k=0)})
and (\ref{eq:R:2_D_Eqn(k=0)}) give us the same solution for $\bt_n$.
\end{remark}

Cases \ref{C:R:3DE:JacRed2} and \ref{C:R:3DE:JacRed3} from the previous section can also be directly verified as before but do not seem to reduce to something simpler.
\end{enumerate}
\subsection{Fixed points of the second order equation}

In anticipation of an asymptotic expansion for the desired quantities and in keeping with what is know in the
classical Jacobi case, we now find the fixed points of the second order equation.

Take $n$ large, fix $k^2,$ and replace $\bt_n$ by $C$
in (\ref{eq:R:2nd_Ord_Diff_Eqn}),
where $C$ depends on $\al$, $\bt$ and $k^2$,
we find that
(\ref{eq:R:2nd_Ord_Diff_Eqn}) becomes a degree $6$ equation in $C$:
\begin{small}
\bea\label{eq:R:2_D:btnConstSolnEq}
0&=&
(4 C-1)^2 (4 {k}^2 C-1) ^2 ({k}^2 -1) ^2 (\alpha+\beta+n) ^4 +\bigg[2 ({k}^2 -1) ^2 (4 {k}^2 C-1) {\alpha}^2 +2 ({k}^2 -1) ^2 (4 C-1) {\beta}^2
\nonumber\\&&
-2 C(8 {k}^2 C-1-{k}^2 ) (16 {k}^4 {C}^2 -4 {k}^4 C-4 {k}^2 C+{k}^4 -{k}^2 +1) \bigg] (4 {k}^2 C-1) (4 C-1) (\alpha+\beta+n) ^2
\nonumber\\&&
+\bigg[
48k^4C^3-12k^2(1+k^2)C^2-\Big((4\al^2-1)k^4-(4\al^2+4\bt^2+1)k^2+4\bt^2-1\Big)C+(\al^2-\bt^2)(k^2-1)
\bigg]^2.
\nonumber\\
\eea
\end{small}\noindent
This is also a degree 4 equation in $n.$ For large $n$, we discard beyond
${\rm O}(n^4)$ to find,
\bea\label{eq:R:2_D:btnConstSolnEq_Large_n}
\left(k^2C-\frac{1}{4}\right)^2\left(C-\frac{1}{4}\right)^2(k^2-1)^2&=&0.
\eea
Hence,
\begin{align}
\label{eq:R:2_D:NLDE(s1)}
\lim_{n\to\infty}\bt_n=\frac{1}{4},&&
\lim_{n\to\infty}\bt_n=\frac{1}{4k^2},
\end{align}
since $0<k^2<1.$ We discard the $\frac{1}{4k^2}$ solution, since from \cite{Szego1939,ChenIsmail1997T}, it is to be expected that $\lim_{n\to\infty}{\sqrt {\bt_n}}=\frac{1}{2},$ 
where $\frac{1}{2}$ is the transfinite diameter\footnote{For a real interval $[A,B]\subset\mathbb{R}$, the transfinite diameter is given by $(B-A)/4$.} of $[-1,1]$.

\section{Large \texorpdfstring{$n$}{n} expansion of  \texorpdfstring{$\bt_n$}{recurrence coefficient} and  \texorpdfstring{$\P_1(n)$}{sub-leading term}}\label{Sec:R:LargeN_Exp}
In this section we use our difference equations (\ref{eq:R:2nd_Ord_Diff_Eqn}) and (\ref{eq:R:2D_Eqn_p1n}) to compute the large $n$ expansion of $\bt_n$ and $\textsf{p}_1(n)$ respectively. Using our difference equations, we find a complete asymptotic expansion for $\bt_n$ and $\P_1(n)$, with explicitly computable constants that only depend on $k^2$, $\al$ and $\bt$. The form of the asymptotic expansion for the following quantities was shown to exist in \cite{KuijlaarsMcVaVan2004} using the Riemann-Hilbert approach. 
 Our difference equation approach enables a very quick calculation of the asymptotic expansion of $\bt_n$ and $\P_1(n)$, with determined constants.
We shall then show in the next two sections that either expansion ($\bt_n$ or $\P_1(n)$) may be used to compute the large $n$ expansion of the Hankel determinant $D_n$, which holds independently of \cite{KuijlaarsMcVaVan2004}.

In the large $n$ limit, it follows from \cite{KuijlaarsMcVaVan2004} that $\bt_n$ has an expansion of the following form:
\bea\label{eq:R:btnExp}
\bt_n&=&a_0(k^2,\alpha,\beta)+\sum\limits_{j=1}^\infty \frac{a_j(k^2,\al,\bt)}{n^j},
\qquad \al>-1,\;\;\bt\in\mathbb{R}, \;\;k^2\in(0,1).
\eea
The following table lists some of the expansion coefficients $a_j$ for $\bt_n$ corresponding to some of the Classical Orthogonal Polynomials.
Our general coefficients $a_j(k^2,\al,\bt)$ should reduce to these values for special values of $k^2$, $\al$ and $\bt$.
\begin{footnotesize}
\begin{table}[!ht]
\caption{The Large $n$ Expansion Coefficients of $\bt_n$ for
Some Classical Orthogonal Polynomials (the higher order coefficients are given in terms of $a_2$).}
\vspace{0.2cm}\small
\begin{tabular*}{\textwidth}{@{\extracolsep{\fill} }cccc}\hline
\label{T:R:LargeNClassicalCoeffs}
Name&Jacobi&Legendre &Chebyshev $1$st/$2$nd \\\hline

$w(x)$&$(1-x^2)^{\al}$&$1$&$(1-x^2)^{\mp\frac{1}{2}}$ \\\hline

$a_0$&$\frac{1}{4}$&$\frac{1}{4}$&$\frac{1}{4}$ \\

$a_1$&$0$&$0$&$0$ \\

$a_2$&$\frac{1-4\al^2}{16}$&$\frac{1}{16}$&$0$ \\

$a_3$&$-2\al a_2$&$0$&$0$ \\

$a_4$&$\frac{a_2}{4}(12\al^2+1)$&$\frac{a_2}{4}$&$0$ \\

$a_5$&$-\al a_2(4\al^2+1)$&$0$&$0$ \\

$a_6$&$\frac{a_2}{16}(80\al^4+40\al^2+1)$&$\frac{a_2}{16}$&$0$ \\\hline
\end{tabular*}
\end{table}
\end{footnotesize}\noindent

We use two methods to study the expansion coefficients $a_j(k^2,\al,\bt)$, where we refrain from showing the $k^2$, $\al$ and
$\bt$ dependence henceforth. First using the second-order difference equation (\ref{eq:R:2nd_Ord_Diff_Eqn}) in Section~\ref{sec:R:2_D_btn_Ln_Exp}, and then using the third-order equation (\ref{eq:R:3rd_Ord_D_Eqn_btn_Intro}) in Section~\ref{sec:LargeNExp_btn_Higher_Order}.

\subsection{Large \texorpdfstring{$n$}{n} expansion of second order difference equation for  \texorpdfstring{$\bt_n$}{recurrence coefficient}}\label{sec:R:2_D_btn_Ln_Exp}
Upon substitution of (\ref{eq:R:btnExp}) into (\ref{eq:R:2nd_Ord_Diff_Eqn}), and taking a large $n$ limit, we have an expression of the form
\bea\label{eq:R:2_D:LnExpansion}
e_{-4}n^4+e_{-3}n^3+\sum\limits_{j=-2}^{\infty}\frac{e_{j}}{n^{j}}&=&0,
\eea
where each $e_j$ depends upon the expansion coefficients $a_j$, $\al$, $\bt$ and $k^2$ (to be precisely determined).
Assuming that in order to satisfy the above equation, each power of $n$ is identically zero, the equation $e_{-4}=0$ gives us
the following relation:
\bea
(k^2-1)^2\left(k^2a_0-\frac{1}{4}\right)^2\left(a_0-\frac{1}{4}\right)^2&=&0.
\eea
This is just (\ref{eq:R:2_D:btnConstSolnEq_Large_n}), as expected.

Even though, $k^2\in(0,1)$, we can still apply  the expansion (\ref{eq:R:2_D:LnExpansion}) to obtain the
large $n$ expansion of $\bt_n$  for $k^2=1$. In the following subsections, we consider two separate cases,
one where $k^2\in(0,1)$ and the other where $k^2=1$.

\subsubsection{Large $n$ expansion of (\ref{eq:R:2nd_Ord_Diff_Eqn}) for $k^2\in(0,1)$}
Given the previous computation, it is natural to examine the large $n$ expansion around $\bt_n=\frac{1}{4}$, and 
so we set $a_0=\frac{1}{4}$. We now state the following lemma:
\begin{lemma}
$\bt_n$ has the following large $n$ expansion in powers of $1/n$:
\bea\label{eq:R:btnExp_Lemma}
\bt_n&=&\frac{1}{4}+\sum\limits_{j=1}^\infty \frac{a_j(k^2)}{n^j},
\qquad \al>-1,\;\;\bt\in\mathbb{R},\;\;k^2\in(0,1),
\eea
where
\begin{small}
\bea
\label{eq:R:Ln:2_D:a1}
a_1(k^2)&=&0,\\
\label{eq:R:Ln:2_D:a2}
a_2(k^2)&=&\frac{1-4\al^2}{16},
\\
\label{eq:R:Ln:2_D:a3}
a_3(k^2)&=&
\frac{(4\al^2-1)}{8} {\bigg[ \al+\beta -\frac{\beta}{\sqrt {1-{k}^2} }\bigg] },
\\
\label{eq:R:Ln:2_D:a4}
a_4(k^2)&=&
{\frac{(1-4\al^2)}{64}}\bigg[ 12(\al+\bt)^2+1-{\frac{24(\alpha+\beta)\bt  }{{\sqrt{1-{k}^2}}}}+\frac{12\bt^2}{1-k^2}
\bigg],\\
\label{eq:R:Ln:2_D:a5}
a_5(k^2)&=&
\frac{(4\alpha^2-1)}{64}
\Bigg[ 4\Big(4 (\alpha+\beta) ^2 +1\Big) (\alpha+\beta)
-\frac{\Big(48(\al+\bt)^2 +4{\alpha}^2 -5\Big) \beta}{\sqrt{1-k^2}}
+{\frac {48{\beta}^2 (\alpha+\beta) }{1-{k}^2}}
\nonumber\\*&&
-  {\frac {(16 {\beta}^2 -4 {\alpha}^2 +9) \beta }{(1-{k}^2)^{\frac{3}{2}} }} \Bigg],
\\
\label{eq:R:Ln:2_D:a6}
a_6(k^2)&=&
\frac{5(1-4{\alpha}^2 )}{64}
\Bigg[4(\alpha+\beta) ^4 +2(\alpha+\beta) ^2 +{\frac{1}{20}}- {\frac {{\beta}^2 (4\al^2-4 {\beta}^2  -9) }{(1-{k}^2 )^2 }}
+{\frac {{\beta}^2 \Big(24(\al+\bt)^2 +4{\alpha}^2 -7 \Big) }{1-{k}^2 }}
\nonumber\\*&&
-\frac{1}{\sqrt{1-k^2}}
\bigg\{\Big(16(\al +{\beta})^2 +4 {\alpha}^2 -5\Big) \beta  (\alpha+\beta)
-{\frac {\beta(\al+\bt) (4\al^2-16\bt^2-9) }{1-{k}^2 }}\bigg\}
\Bigg],
\eea
\end{small}\noindent
are the first few terms, with more terms easily computable.

We can check that the limits of each coefficient as $k^2\to 0$ reduces to those listed in Table~\ref{T:R:LargeNClassicalCoeffs}.

\end{lemma}
\begin{proof}
Setting $a_0=\frac{1}{4}$ then leads to the quantity
\bea
e_{-3}&=&
128(k^2-1)^2 \left({k}^2 a_{0}-\frac{1}{4}\right)\left(a_{0}-\frac{1}{4}\right)
\bigg[\frac{1}{2} (4 a_{0}-1) (4 {k}^2 a_{0}-1) 
(\alpha+\beta) -a_{1 }({k}^2 +1-8 {k}^2 a_{0})\bigg],
\qquad\;\;
\eea
vanishing identically. The equation $e_{-2}=0$ then gives rise to
\bea\label{eq:R:2_D:Ln:e_m2=0}
a_1^2(k^2-1)^4&=&0.
\eea
Recalling that $k^2\in(0,1)$, we obtain
\bea\label{eq:R:Ln:2_D:a1_Proof}
a_1&=&0.
\eea

With $a_0=\frac{1}{4}$ and $a_1=0$, we find that $e_{-1}$ vanishes identically, while the equation $e_0=0$ gives us
\bea\label{eq:R:2_D:Ln:e0=0}
(k^2-1)^4\left(4a_2-\frac{1}{4}+\al^2\right)^2=0.
\eea
Solving for $a_2$, we find that
\bea
\label{eq:R:Ln:2_D:a2_Proof}
a_2&=&\frac{1-4\al^2}{16},
\eea
is independent of $k^2$, {\it the same} as that $a_2$ in Table~\ref{T:R:LargeNClassicalCoeffs}. 

With the above values for $a_0$, $a_1$ and $a_2$, we find that $e_1$ vanishes identically, while setting $e_2=0$ gives
\bea\label{eq:R:2_D:Ln:e2=0}
(k^2-1)^3
\bigg[\Big( (4\al^2-1)(\al+\bt)-8a_3\Big)^2 {k}^2 
-\Big(\al(4\al^2-1) -8 a_{3}\Big) \Big((4\al^2-1)(\al+2\bt)-8 a_{3}\Big) \bigg]&=&0.
\qquad
\eea
Solving for $a_3$ in the above equation gives
\bea\label{eq:R:Ln:2_D:a3pm}
a_3&=&
\frac{1}{8} {\frac {\Big( (\al+\beta)  ({k}^2-1)\pm\beta \sqrt {1-{k}^2 }\Big) (4 {\alpha}^2 -1) }{{k}^2 -1}}.
\eea
In order to choose the correct sign for $a_3$, we take $k^2\to 0$ to find
\bea
a_3&=&\frac{\al(4\al^2-1)}{8},\quad(+\;\text{sign}),\\
a_3&=&\frac{(\al+2\bt)(4\al^2-1)}{8},\quad(-\;\text{sign}).
\eea
Comparing the above values for $a_3$ to the classical values of $a_3$  in Table~\ref{T:R:LargeNClassicalCoeffs}, 
we find that we must take the positive square root sign in
(\ref{eq:R:Ln:2_D:a3pm}), and arrive at
\bea\label{eq:R:Ln:2_D:a3_Proof}
a_3&=&
\frac{(4\al^2-1)}{8} {\bigg[ \al+\beta -\frac{\beta}{\sqrt {1-{k}^2} }\bigg]  }.
\eea
This procedure easily extends to higher coefficients, from which we find $a_4$, $a_5$ and $a_6$.
\end{proof}  
Equations (\ref{eq:R:Ln:2_D:a3})--(\ref{eq:R:Ln:2_D:a6}) are singular at $k^2=1$. However, it is still possible to obtain 
the large $n$ expansion of $\bt_n$ at $k^2=1$ using equation (\ref{eq:R:2_D:LnExpansion}). As we shall see, the crucial difference 
in the case where $k^2=1$ in comparison with the case where $k^2\neq1$ is that  $a_2$ remains undetermined. The higher coefficients 
$a_3$, $a_4$ and so on are given in terms of $a_2$.

\subsubsection{Large $n$ expansion of (\ref{eq:R:2nd_Ord_Diff_Eqn}) at $k^2=1$}
Setting $k^2=1$ in (\ref{eq:R:2_D:LnExpansion}), $e_{-4}$ and $e_{-3}$ vanish identically, while the equation $e_{-2}=0$ gives
\bea
a_0(4a_0-1)^5&=&0.
\eea
Choosing $a_0=\frac{1}{4}$, we find that $e_{-1}$, $e_0$, $e_1$ and $e_2$ vanish identically. The equation $e_3=0$ implies that
$
a_1=0.
$
Now, setting $a_0=\frac{1}{4}$ and $a_1=0$, $e_4$, $e_5,\dots,e_9$ vanish identically, and the equation $e_{10}=0$ is equivalent to
$
a_3=-2(\al+\bt)a_2.
$

This procedure easily extends to higher coefficients, for example,
\bea
a_4&=&
\frac{a_2}{4}\Big[12(\al+\bt)^2+1\Big],\\
a_5&=&
-(\al+\bt) a_2\Big[4(\al+\bt)^2+1\Big].
\eea
We see that the above values of $a_4$ and $a_5$ agrees with the those in Table~\ref{T:R:LargeNClassicalCoeffs}.

\subsection{Large \texorpdfstring{$n$}{n} expansion of higher order difference equations for \texorpdfstring{$\bt_n$}{recurrence coefficient}}\label{sec:LargeNExp_btn_Higher_Order}
We can also substitute  (\ref{eq:R:btnExp}) into our third order difference equation for $\bt_n$, (\ref{eq:R:3rd_Ord_D_Eqn_btn_Intro}), and take a large $n$ limit.
Using the same procedure 
that we have used in Section~\ref{sec:R:2_D_btn_Ln_Exp}, having set $a_0(k^2)=\frac{1}{4}$, the first few expansion coefficients $a_j(k^2)$, $j=1,2,\dots,5$ are found to be
\bea
\label{eq:R:Ln:3_D:a0}
a_0(k^2)&=&\frac{1}{4},\\
\label{eq:R:Ln:3_D:a1}
a_1(k^2)&=&0,\\
\label{eq:R:Ln:3_D:a2}
a_2(k^2)&=&a_2(k^2),\\
\label{eq:R:Ln:3_D:a3}
a_3(k^2)&=&\frac{a_2\sqrt{(16a_2+4\al^2-1)k^2-4\bt^2}}{\sqrt{k^2-1}}-2(\al+\bt)a_2,\\
\label{eq:R:Ln:3_D:a4}
a_4(k^2)&=& 3a_2(\al+\bt)^2+\frac{3a_2}{k^2-1}\left[\left(4a_2+\al^2-\frac{1}{6}\right)k^2-\bt^2-\frac{1}{12}\right]
\nonumber\\*
&&-\frac{3a_2(\al+\bt)\sqrt{(16a_2+4\al^2-1)k^2-4\bt^2}}{\sqrt{k^2-1}},\\
\label{eq:R:Ln:3_D:a5}
a_5(k^2)&=& -4a_2(\al+\bt)^3-\frac{4a_2(\al+\bt)}{k^2-1}\left[\left(12a_2+3\al^2-\frac{1}{2}\right)k^2-3\bt^2-\frac{1}{4}\right]
\nonumber\\*&&
+\frac{6a_2(\al+\bt)^2\sqrt{(16a_2+4\al^2-1)k^2-4\bt^2}}{\sqrt{k^2-1}}
\nonumber\\*&&
+\frac{a_2}{(k^2-1)^{3/2}}\bigg[\big(6a_2+2\al^2-1\big)k^2-2\bt^2-\frac{1}{2}\bigg]\sqrt{(16a_2+4\al^2-1)k^2-4\bt^2}.
\eea
The crucial difference 
in the case of using the third order difference equation (\ref{eq:R:3rd_Ord_D_Eqn_btn_Intro}) in comparison with the second order difference equation (\ref{eq:R:2nd_Ord_Diff_Eqn}) is that  $a_2$ remains undetermined. The higher coefficients 
$a_3$, $a_4$ and so on are given in terms of $a_2$.

Similarly, we can use our fourth order difference equation for $\bt_n$, (\ref{eq:R:Rees4ODEqn}), to calculate the large $n$ expansion of $\bt_n$. In this case we would find that $a_2$ and $a_3$ remains undetermined, with the higher coefficients $a_4$, $a_5$ and so on given in terms of $a_2$ and $a_3$. 

By substituting in for $a_2$ using (\ref{eq:R:Ln:2_D:a2}), we find that $a_3$, $a_4$ and $a_5$, given by (\ref{eq:R:Ln:3_D:a3}), (\ref{eq:R:Ln:3_D:a4}) and (\ref{eq:R:Ln:3_D:a5}) respectively, precisely match with the coefficients $a_3$, $a_4$ and $a_5$ obtained from the expansion of the second order difference equation, given by (\ref{eq:R:Ln:2_D:a3}), (\ref{eq:R:Ln:2_D:a4}) and (\ref{eq:R:Ln:2_D:a5}) respectively.

\subsection{Large \texorpdfstring{$n$}{n} expansion of second order difference equation for \texorpdfstring{$\P_1(n)$}{sub-leading term}}\label{sec:R:2_D_p1n_Ln_Exp}
Using (\ref{eq:R:2D_Eqn_p1n}), we can find the asymptotic expansion of $\P_1(n)$ in powers of $1/n$. We state this in the following lemma:
\begin{lemma}
$\P_1(n)$ has the following asymptotic expansion, for large $n,$
\bea\label{eq:R:p1nExp}
\P_1(n)&=&nb_{-1}(k^2)+\sum\limits_{j=0}^\infty \frac{b_j(k^2)}{n^j},
\qquad \al>-1,\;\;\bt\in\mathbb{R},\;\;k^2\in(0,1).
\eea
The coefficients $b_j(k^2)$, $j=-1,0,1,\dots,$ (which also depend on $\al$ and $\bt$) are explicitly computable. The first few up to $b_5$ are
\bea
\label{eq:R:Ln:2_D_p1n:bm1}b_{-1}(k^2)&=&-\frac{1}{4},\\
\label{eq:R:Ln:2_D_p1n:b0}b_0(k^2)&=&
\frac{1}{4}(\al-\bt)+\frac{1}{8}+\frac{\bt(1-\sqrt{1-k^2})}{2k^2},\\
\label{eq:R:Ln:2_D_p1n:b1}b_1(k^2)&=&\frac{1-4\al^2}{16},\\
\label{eq:R:Ln:2_D_p1n:b2}b_2(k^2)&=&\frac{4\al^2-1}{16}\bigg[\al+\bt-\frac{1}{2}-\frac{\bt}{\sqrt{1-k^2}}\bigg],\\
\label{eq:R:Ln:2_D_p1n:b3}b_3(k^2)&=&
\frac{1-4\al^2}{16}\bigg[\al+\bt-\frac{1}{2}-\frac{\bt}{\sqrt{1-k^2}}\bigg]^2,\\
\label{eq:R:Ln:2_D_p1n:b4}b_4(k^2)&=&
\frac{4\al^2-1}{16}\Bigg[\bigg(\al+\bt-\frac{1}{2}-\frac{\bt}{\sqrt{1-k^2}}\bigg)^3+\frac{(4\al^2-9)\bt k^2}{16(1-k^2)^{3/2}}\Bigg],\\
\label{eq:R:Ln:2_D_p1n:b5}b_5(k^2)&=&
\frac{1-4\al^2}{16}
\bigg(\al+\bt-\frac{1}{2}-\frac{\bt}{\sqrt{1-k^2}}\bigg)
\Bigg[\bigg(\al+\bt-\frac{1}{2}-\frac{\bt}{\sqrt{1-k^2}}\bigg)^3+\frac{(4\al^2-9)\bt k^2}{4(1-k^2)^{3/2}}\Bigg].
\eea
\end{lemma}
\begin{proof}
We substitute (\ref{eq:R:p1nExp}) into our second order difference equation for $\P_1(n)$, (\ref{eq:R:2D_Eqn_p1n}), 
which gives rise to
\bea
l_{-2}n^2+l_{-1}n+\sum\limits_{j=0}^\infty \frac{l_j}{n^j}=0,
\eea
in a large $n$ expansion. The coefficients $l_j$ depend on the expansion coefficients $b_j$, $\al$, $\bt$ and $k^2$. We then use  the same procedure
as in Section~\ref{sec:R:2_D_btn_Ln_Exp} to find 
the first few expansion coefficients $b_j$, $j=-1,0,1,\dots,5$ for $k^2\in(0,1)$. \end{proof}
For consistency, we can replace $\P_1(n)$ and $\P_1(n+1)$ in $\bt_n=\P_1(n)-\P_1(n+1)$ by their large $n$ expansion~(\ref{eq:R:p1nExp}), where we set $n=n$ and $n\to n+1$ respectively. Then, for large $n$,
the recurrence coefficient $\beta_n$ admits the following asymptotic expansion,
\bea
\bt_n&=& 
\frac{1}{4}+\frac{1-4\al^2}{16n^2}
+\frac{(4\al^2-1)}{8n^3} 
\bigg[ \al+\beta -\frac{\beta}{\sqrt {1-{k}^2} }\bigg]+{\rm{O}}\left(\frac{1}{n^4}\right).
\eea
We find that this agrees precisely with the expansion for $\bt_n$ obtained in Section~\ref{sec:R:2_D_btn_Ln_Exp}.

\section{Large \texorpdfstring{$n$}{n} expansion of the Hankel determinant}\label{Sec:R:Hankel_LargeN}
In this section we prove Theorem~\ref{thm:R:Dn_LargeNExp_Intro} and compute the large $n$ expansion of our
Hankel determinant
\bea\label{eq:R:Dn_Def_LargeN_Intro}
D_n[w(\cdot,k^2)]=\det\left(\int\limits_{-1}^{1}x^{i+j}w(x,k^2)dx\right)_{i,j=0}^{n-1},
\eea
and we recall for the Reader our weight function,
\bea
w(x,k^2)&=&(1-x^2)^{\al}(1-k^2x^2)^{\bt},\qquad x\in[-1,1],\;\;\al>-1,\;\;\bt\in\mathbb{R},\;\;k^2\in(0,1).\nonumber
\eea
\subsection{Leading term of asymptotic expansion}
We describe here a ``linear statistics" approach to compute $D_n[w(\cdot,k^2)]$ for large $n$, which gives the leading term
of the asymptotic expansion. As this can be obtained using existing results, obtained in \cite{ChenLawrence1997} and
\cite{BasorChen2005},  we will be brief. The idea is to re-write our Hankel determinant as
$$
D_n[w^{(\al,\al)}]\cdot\frac{D_n[w(\cdot,k^2)]}{D_n[w^{(\al,\al)}]},
$$
where $D_n[w^{(\al,\al)}]$ is the Hankel determinant generated by a special case of the Jacobi weight,
$w^{(\al,\al)}(x)=(1-x^2)^{\al},\;\;x\in[-1,1],\;\;\al>-1$. A closed-form expression for this maybe obtained
\cite{BasorChen2005}; the leading term of its large $n$ expansion reads,
$$
D_n[w^{(\al,\al)}]\sim (2\pi)^n\;n^{\al^2-1/4}\;2^{-n(n+\al)}\:\frac{(2\pi)^{\al}\pi^{1/2}[G(1/2)]^2}{2^{2\al^2}}[G(1+\al)]^2.
$$
A straightforward application of the results obtained in \cite{ChenLawrence1997} gives the following leading term for large $n:$
$$
\frac{D_n[w(\cdot,k^2)]}{D_n[w^{(\al,\al)}]}\sim \left(\frac{1+\sqrt{1-k^2}}{2}\right)^{2\bt\;n}
\;\left(\frac{1+\sqrt{1-k^2}}{2\sqrt{1-k^2}}\right)^{2\al\bt}\;\left[\frac{(1+\sqrt{1-k^2})^2}{4\sqrt{1-k^2}}\right]^{\bt^2}.
$$
In the above constant the function $G$ refers to the Barnes $G$-function, an entire function that satisfies the recurrence relation
$$
G(z+1)=\Gamma(z)G(z),\qquad G(1)=1.
$$

An alternative approach is described below. This is done  by reducing (\ref{eq:R:Dn_Def_LargeN_Intro}) to another expansion of 
a certain Toeplitz plus Hankel determinant and
then using \cite{DeiftItsKrasovsky2011}. 

Let us define, for $b(x)$ an even function on~$[-1,1],$ the following determinant
\bea
\det H_n[b]=\det\left(\frac{1}{\pi}\int\limits_{-1}^{1}\!b(x)(2x)^{j+k}\,dx\right)_{j,k=0}^{n-1}.
\eea
An easy computation shows that
\bea
\det H_n[b]=\frac{2^{n(n-1)}}{\pi^n}\det\left(\int\limits_{-1}^{1}\!x^{j+k}b(x)\,dx\right)_{j,k=0}^{n-1}.
\eea
Therefore we would like to compute this determinant with $b(x)=w(x,k^2).$  From the computations in \cite{BasorChenEhrhardt}, we have that
\bea
\det H_n[b]=2^{-2n\al}\det (T_n(a)+H_n(a)),
\eea
where
\bea
a(\theta)=(1-k^2\cos^2\theta)^{\bt}(2-2\cos\theta)^{\al+1/2}(2+2\cos\theta)^{\al-1/2},\qquad0<\theta\leq\pi.
\eea
The above is a finite determinant of Toeplitz plus Hankel matrices.
So the large $n$ behavior of $D_n[w(\cdot,k^2)]$ can be obtained from the large $n$
behavior of $\det H_n[w(\cdot,k^2)],$ up to the factor
$$
2^{-n(n-1)}\pi^n.
$$
To compute the determinant of the Toeplitz plus Hankel matrices, and since  $a$  is
even in $\theta$ we make use of a result presented in \cite{DeiftItsKrasovsky2011} 
to obtain,
\bea
\det (T_n(a)+H_n(a))\sim {\rm e}^{nG}\;n^{p}\;E,
\eea
where $G,$ $p$ and $E$ are expressed in terms of the parameters of our weight, $\al,\;\bt\;$ and $k^2$.

We have
\bea
{\rm e}^{nG}:=\left(\frac{1+\sqrt{1-k^2}}{2}\right)^{2\bt\;n},
\eea
\bea
p:=\al^2-1/4,
\eea
and
\bea\label{eq:R:E_Leading_Term_Dn}
E:=\frac{(2\pi)^{\al}\pi^{1/2}[G(1/2)]^2}{2^{2\al^2}\;[G(1+\alpha)]^2}\;
\left(\frac{1+\sqrt{1-k^2}}{2\sqrt{1-k^2}}\right)^{2\al\bt}
\left[\frac{(1+\sqrt{1-k^2})^{2}}{4\sqrt{1-k^{2}}}\right]^{\bt^2}.
\eea
This means that, as $n\rightarrow\infty$,
\bea\label{eq:dnlargen}
D_n[w(\cdot,k^2)]\sim (2\pi)^n\;n^{\al^2-1/4}\;2^{-n(n+2\al)}\;\left(\frac{1+\sqrt{1-k^2}}{2}\right)^{2\bt\;n}\;E.
\eea
\subsection{Correction terms of asymptotic expansion}
We now supply the ``correction terms" using the large $n$ expansion of $\bt_n.$ Recall that
$$
\bt_n=\frac{1}{4} +\frac{a_2(\al,\bt,k^2)}{n^2}+\frac{a_3(\al,\bt,k^2)}{n^3}+{\rm O}\left(\frac{1}{n^4}\right)
$$
as $n\rightarrow\infty.$ Let
$$
F_n:=-\log\;D_n,
$$
be the ``Free Energy". Since
$$
\bt_n=\frac{D_{n+1}D_{n-1}}{D_n^2},
$$
we have
\bea\label{eq:R:-logbtn_2_D_Eqn_Fn}
-\log\bt_n=F_{n+1}-F_n - (F_n- F_{n-1}).
\eea
Following \cite{ChenIsmail1997T,ChenIsmailVanAssche1998}, for large $n$, we can approximate the above second order difference by
\bea\label{eq:R:CF:logbtn(F)}
-\log\bt_n&=&\frac{\partial^2 F_n}{\partial n^2}+\frac{1}{12}\frac{\partial^4 F_n}{\partial n^4}+{\rm O}\left(\frac{\partial^6 F_n}{\partial n^6}\right).
\eea
We assume that $F_n$ has the following expansion, for large $n,$
\bea\label{eq:R:FnExp}
F_n&=&C(k^2,\al,\bt)\log n+\sum\limits_{j=-2}^\infty \frac{c_j(k^2,\al,\bt)}{n^j},
\qquad \al>-1,\;\;\bt\in\mathbb{R},\;\;k^2\in(0,1).
\eea
Ignoring ${\rm O}\left(\frac{\partial^6 F_n}{\partial n^6}\right)$ terms in (\ref{eq:R:CF:logbtn(F)}), we now substitute (\ref{eq:R:FnExp}) and the asymptotic expansion of $\bt_n$ into (\ref{eq:R:CF:logbtn(F)}).
This give rise to
\bea
m_{0}+\sum\limits_{j=1}^\infty \frac{m_j}{n^j}=0,
\eea
in a large $n$ expansion. The coefficients $m_j$ depends on the expansion coefficients $C$, $c_j$ and $a_j(k^2,\al,\bt)$. Using the same procedure
that we have used in Section \ref{sec:R:2_D_btn_Ln_Exp},
the first few expansion coefficients $C$ and $c_j$, $j=-2,-1,0,\dots,3$ are found to be
\begin{align}
C&= 4a_2,&
c_{-2}&= \log 2,&
c_{-1}&=c_{-1},&
c_{0}&=c_{0},&
\nonumber\\&&
c_{1}&=-2a_3,&
c_{2}&=-\frac{2a_4-a_2(4a_2+1)}{3},&
c_{3}&=-\frac{a_5-a_3(4a_2+1)}{3},
\end{align}
where $a_2(k^2,\alpha,\beta)$---$a_5(k^2,\alpha,\beta)$ can be found in Section \ref{sec:R:2_D_btn_Ln_Exp}. 

Using this method, the expansion for the Free Energy{\textemdash}and hence the Hankel determinant or the Partition Function, $D_n$,{\textemdash}contains two undetermined constants $c_{-1}$ and $c_0$, which can only be determined from (\ref{eq:dnlargen}). 
Incorporating the leading term of the large $n$ expansion from (\ref{eq:dnlargen}), the complete asymptotic expansions reads,
\bea\label{eq:R:Dn_Ln_Exp_TH}
D_n[w(\cdot,k^2)]&\sim&
\exp\bigg[-n^2\log\;2 + c_{-1}\;n - 4a_2 \log\;n + c_{0}+\frac{2a_3}{n}
\nonumber\\&&
\qquad + \frac{2a_4-a_2(4a_2+1)}{3n^2}
+\frac{a_5-a_3(4a_2+1)}{3n^3}+{\rm O}\left(\frac{1}{n^4}\right)\;\bigg],
\eea
where
\bea
{\rm e}^{c_0}=E,
\eea
and
\bea
c_{-1}=2\bt\;\log\left(\frac{1+\sqrt{1-k^2}}{2}\right)+\log\pi+(1-2\al)\:\log2.
\eea

To calculate higher order terms in (\ref{eq:R:Dn_Ln_Exp_TH}), we would require ${\rm O}\left(\frac{\partial^6 F_n}{\partial n^6}\right)$ correction terms in (\ref{eq:R:CF:logbtn(F)}). Notice that we could have easily approximated the second order difference in (\ref{eq:R:-logbtn_2_D_Eqn_Fn})
by
$
\frac{\partial^2F_n}{\partial n^2},
$
for large $n.$ We could then integrate $\log\bt_n$ twice with respect to $n$ to recover the $F_n$ and hence
$D_n$, in a large $n$ expansion. However, this would mean that our asymptotic expansion for $D_n$ would only be valid up to and including the ${\rm O}\left(\frac{1}{n}\right)$ term. This is because we obtain $c_2$ in terms of the coefficients $a_2$ and $a_4$ from the equation $m_4=0$, which has contributions from the ${\rm O}\left(\frac{\partial^4 F_n}{\partial n^4}\right)$ term.

\section{Alternative computation of large \texorpdfstring{$n$}{n} expansion for \texorpdfstring{$D_n$}{Hankel determinant}}\label{Sec:R:Hankel_LargeN_Toda}
In this section we give an alternative computation of the large $n$ expansion of our
Hankel determinant
\bea\label{eq:R:Dn_Def_LargeN_Intro_2}
D_n[w(\cdot,k^2)]=\det\left(\int\limits_{-1}^{1}x^{i+j}
(1-x^2)^{\al}(1-k^2x^2)^{\bt}dx\right)_{i,j=0}^{n-1}, \qquad
\al>-1,\;\; \bt\in\mathbb{R},\;\; k^2\in(0,1),
\eea
using results obtained from the ladder operators. We combine the large $n$ expansion for $\textsf{p}_1(n)$ with Toda-type time-evolution equations satisfied by $D_n$ (to be presented in due course).

We restate Theorem \ref{thm:R:Dn_LargeNExp_Intro} as the following:
\begin{theorem}\label{thm:Large_n_Exp_Dn}
The ratio $D_n[w(\cdot,k^2)]/D_n[w^{(\al,\al)}(\cdot)]$ has an asymptotic expansion in $n$ of the form 
\bea\label{eq:R:Dn(k2)/D_n(0)_Large_n}
\frac{D_n[w(\cdot,k^2)]}{D_n[w^{(\al,\al)}(\cdot)]}&=&
\frac{\Big(1+\sqrt{1-k^2}\Big)^{2\bt(n+\al+\bt)}}{2^{2\bt(n+\al+\bt)}\Big(1-k^2\Big)^{\bt(\al+\bt/2)}}
\exp
\Bigg[
\frac{(4\al^2-1)\bt}{4n}\bigg(1-\frac{1}{\sqrt{1-k^2}}\bigg)
\nonumber\\*&&
\qquad
-\frac{(4\al^2-1)\bt}{8n^2}\bigg(2\al+\bt-\frac{2(\al+\bt)}{\sqrt{1-k^2}}+\frac{\bt}{1-k^2}\bigg)+{\rm O}\left(\frac{1}{n^3}\right)
\Bigg].
\eea
Hence, the Hankel determinant $D_n[w(\cdot,k^2)]$ has the following asymptotic expansion in $n$:
\bea\label{eq:R:Dn(k2)_Large_n}
D_n[w(\cdot,k^2)]
&=&
E\;n^{\al^2-1/4}\; 2^{-n(n+2\al)}\;(2\pi)^{n}\;\left(\frac{1+\sqrt{1-k^2}}{2}\right)^{2\bt n}
\nonumber\\*&&
\times\;\exp
\Bigg[
\frac{2a_3}{n}
+\frac{2a_4-a_2\Big(4a_2+1\Big)}{3n^2}
+\frac{a_5-a_3\Big(4a_2+1\Big)}{3n^3}+{\rm O}\left(\frac{1}{n^4}\right)
\Bigg],
\eea
where the $n$-independent constant $E$ is given by
\bea
E&:=&
\frac{(2\pi)^{\al}\pi^{1/2}[G(1/2)]^2}{2^{2(\al^2+\al\bt+\bt^2)}[G(1+\al)]^2}\cdot\frac{\left(1+\sqrt{1-k^2}\right)^{2\bt(\al+\bt)}}{(1-k^2)^{\bt(\al+\bt/2)}}.
\eea
In the above, the coefficients $a_2(k^2,\alpha,\beta)$---$a_5(k^2,\alpha,\beta)$ can be found in Section \ref{sec:R:2_D_btn_Ln_Exp}. 
\end{theorem}
Note that the large $n$ expansion for $D_n[w(\cdot,k^2)]$, equation (\ref{eq:R:Dn(k2)_Large_n}), precisely agrees with the result obtained from Section~\ref{Sec:R:Hankel_LargeN}, equation (\ref{eq:R:Dn_Ln_Exp_TH}).

To prove Theorem~\ref{thm:Large_n_Exp_Dn}, we first need to find the Toda-type time-evolution equations for our Hankel determinant.
\subsection{Toda evolution}
In this section, $n$ is kept fixed while we vary the parameter $k^2$ in the weight function (\ref{defn:R:w(xk2)}). The other parameters $\al$ and $\bt$ are also kept fixed.

Differentiating the definition of $h_n$, equation (\ref{def:hn}), w.r.t. $k^2$, and integrating by parts, we have
\bea
\label{eq:R:dk2loghn}k^2\frac{d}{dk^2}\log h_n&=&
R_n-n-\al-\frac{1}{2}.
\eea
From (\ref{def:betan}), i.e. $\bt_n=h_n/h_{n-1}$, this implies that
\bea
\label{eq:R:dk2logbtn}k^2\frac{d}{dk^2}\log \bt_n&=&R_n-R_{n-1}-1.
\eea

\subsection{Toda evolution of Hankel determinant}
In this section, we describe how we can express the logarithmic derivative of the Hankel determinant $D_n$ in terms of $\P_1(n)$ and $\P_1(n+1)$. This is established in the following lemma:

\begin{lemma}\label{lem:R:Hn(p1np1n+1)}
We define the quantity $H_n(k^2)$ through the Hankel determinant $D_n(k^2):=D_n[w(\cdot,k^2)]$ as
\bea\label{eq:R:Hn_def}
H_n(k^2)&=&k^2(k^2-1)\frac{d}{dk^2}\log D_n(k^2).
\eea
Then we may express $H_n(k^2)$ in terms of $\P_1(n)$ and $\P_1(n+1)$ as
\bea\label{eq:R:Hn(p1np1n+1)}
H_n(k^2)&=&
-k^2\left(\al+\bt+n-\frac{1}{2}\right)\P_{1}(n)-k^2\left(\al+\bt+n+\frac{1}{2}\right)\P_1(n+1)-\frac{n^2k^2}{2}.\qquad\;\;
\eea
\end{lemma}
\begin{proof}
Using (\ref{def:prodhn}), i.e. $D_n=\prod\limits_{j=0}^{n-1}\!h_j$, and (\ref{eq:R:dk2loghn}), we can write the logarithmic derivative of the Hankel determinant $D_n$ as
\bea
H_n(k^2)
&=&k^2(k^2-1)\frac{d}{dk^2}\sum\limits_{j=0}^{n-1}\log h_j,
\nonumber\\
&=&(k^2-1)\sum_{j=0}^{n-1}R_j-n\left(\al+\frac{n}{2}\right)(k^2-1).
\eea
We replace $\sum\limits_{j=0}^{n-1}R_j$ using (\ref{eq:R:SumR(btnrn)}) to arrive at
\begin{footnotesize}
\bea
\frac{H_n(k^2)}{2}&=&
-(k^2-1)(\alpha+\beta+n)r_n
-k^2\left(\al+\bt+n+\frac{1}{2}\right)\left(\al+\bt+n-\frac{1}{2}\right)\beta_n
+\frac{n}{2}\left(\alpha+\beta+n-\frac{nk^2}{2}\right),\\
&=&\frac{k^2}{2}\left(\al+\bt+n+\frac{1}{2}\right)\bt_n-k^2(\al+\bt+n)\P_1(n)-\frac{n^2k^2}{4},
\eea
\end{footnotesize}\noindent
where the second equality follows from eliminating $r_n$ in favor of $\bt_n$ and $\P_1(n)$ using (\ref{eq:R:p1n(rnbtn)}). Finally, eliminating $\bt_n$ using the identity $\bt_n=\P_1(n)-\P_1(n+1)$ leads to equation (\ref{eq:R:Hn(p1np1n+1)}), completing the proof.
\end{proof}
Combining Lemma~\ref{lem:R:Hn(p1np1n+1)} with the large $n$ expansion for $\textsf{p}_1(n)$, (\ref{eq:R:p1nExp}), will allow us to compute the large $n$ expansion for $D_n$.
\subsection{Proof of Theorem~\ref{thm:Large_n_Exp_Dn}}
\begin{proof}
To calculate the large $n$ expansion of $D_n(k^2)$, the idea is to re-write our Hankel determinant as
$$
D_n(0)\cdot\frac{D_n(k^2)}{D_n(0)},
$$
where $D_n(0)$ is the Hankel determinant generated
generated by a special case of the Jacobi weight,
$w^{(\al,\al)}(x)=(1-x^2)^{\al},\;\;x\in[-1,1],\;\;\al>-1$. The leading order term for the expansion of $D_n(0)$ can be found in
\cite{BasorChen2005}. However, we require higher order terms in $n$, which are quite easy to calculate.

For monic Jacobi polynomials orthogonal to the weight $w^{(\al,\al)}(x)$, it is well-known that \cite{KoekoekLeskySwart2000} (see \cite{ChenIsmail2005} for a derivation using the ladder operator approach)
\bea
h_n(0)&=&2^{2n+2\al+1}\;\frac{[\Gamma(n+\al+1)]^2\Gamma(n+2\al+1)\Gamma(n+1)}{\Gamma(2n+2\al+1)\Gamma(2n+2\al+2)}.
\eea
Hence it follows from (\ref{def:betan}) that
\bea\label{eq:R:D_n(0)_Exact}
D_n(0)
&=&
\frac{2^{n(n+2\al)}}{[G(\al+1)]^2}
\cdot\frac{G(n+1)[G(n+\al+1)]^2G(n+2\al+1)}{G(2n+2\al+1)}.
\eea
The asymptotics of the Barnes $G$-function is well understood \cite[p.~284]{Barnes1899}. For $\la\in\mathbb{C}$ such that $\vert\la\vert$ is finite, we have that
\bea\label{def:Barnes_G_asympt}
G(n+\lambda+1)&=&
(2\pi)^{(n+\lambda)/2}\;n^{{(n+\lambda)^2/2}-1/12}\;[G(1/2)]^{2/3}\;\pi^{\frac{1}{6}}\;2^{-\frac{1}{36}}
\;{\rm e}^{-\frac{3n^2}{4}-n{\lambda}}
\nonumber\\&&
\times\exp
\Bigg[
\frac{{\lambda}(2{\lambda}^2-1)}{12n}-\frac{10{\lambda}^4+10{\lambda}^2-1}{240n^2}
+\frac{{\lambda}(6{\lambda}^2-10{\lambda}^2+3)}{360n^3}
\nonumber\\&&\qquad\qquad
-\frac{42{\lambda}^6-105{\lambda}^4+63{\lambda}^2-5}{5040n^4}+{\rm O}\left(\frac{1}{n^5}\right)
\Bigg].
\eea
Applying the above formula with $\la=0$, $\la=\al$ and $\la=2\al$ for the numerator of (\ref{eq:R:D_n(0)_Exact}), and with $n\to2n$ and $\la=2\al$ for the denominator; for large $n$, we find that
\begin{small}
\bea\label{eq:R:D_n(0)_Large_n}
D_n(0)&=&D_n[w^{(\al,\al)}(\cdot)]
\nonumber\\*
&=&(2\pi)^{n+\al}\;n^{\al^2-1/4}\;2^{-n(n+2\al)}\;\frac{\pi^{\frac{1}{2}}[G(1/2)]^2}{2^{2\al^2}[G(\al+1)]^2}
\nonumber\\*&&
\times\exp
\Bigg[
\frac{2a_3(0,\al,\bt)}{n}+\frac{2a_4(0,\al,\bt)-a_2(0,\al,\bt)\Big(4a_2(0,\al,\bt)+1\Big)}{3n^2}
\nonumber\\&&\qquad\quad
+\frac{a_5(0,\al,\bt)-a_3(0,\al,\bt)\Big(4a_2(0,\al,\bt)+1\Big)}{3n^3}
+{\rm O}\left(\frac{1}{n^4}\right)
\Bigg].
\eea
\end{small}
\!\!In the above, the coefficients $a_2(k^2,\al,\bt)$---$a_5(k^2,\al,\bt)$ are evaluated at $k^2=0$ (they can be found in Section~\ref{sec:R:2_D_btn_Ln_Exp}).

Now we proceed to calculate the large $n$ expansion of the ratio $D_n(k^2)/D_n(0)$. Using the definition of~$H_n(k^2)$, (\ref{eq:R:Hn_def}), the Hankel determinant $D_n(k^2)$ has the following integral representation:
\bea\label{eq:R:Ln:DnIntRep}
\frac{D_n(k^2)}{D_n(0)}&=&\exp\left(\int\limits_0^{k^2}\!\frac{H_n(k^2)}{k^2(k^2-1)}\,dk^2\right),
\eea
where $H_n(k^2)$ is related to $\P_1(n)$ and $\P_1(n+1)$ through equation~(\ref{eq:R:Hn(p1np1n+1)}) (Lemma~\ref{lem:R:Hn(p1np1n+1)}).

We replace $\P_1(n)$ and $\P_1(n+1)$ in (\ref{eq:R:Hn(p1np1n+1)}) by their large $n$ expansion (\ref{eq:R:p1nExp}), where we set $n=n$ and $n\to n+1$ respectively. Expanding again in the large $n$ limit, we find that $H_n(k^2)$ admits the following expansion in powers of $1/n$:
\bea
H_n(k^2)&=&
\bt\sqrt{1-k^2}(1-\sqrt{1-k^2})n-
\bt(\al+\bt)(1-\sqrt{1-k^2})+\frac{k^2\bt^2}{2}
\nonumber\\&&
+\frac{k^2\bt(4\al^2-1)}{8n\sqrt{1-k^2}}
-\frac{k^2\bt(4\al^2-1)}{8n^2\sqrt{1-k^2}}\left(\al+\bt-\frac{\bt}{\sqrt{1-k^2}}\right)
+{\rm O}\left(\frac{1}{n^3}\right).
\eea
Substituting the above expansion for $H_n(k^2)$ into equation (\ref{eq:R:Ln:DnIntRep}), and integrating with respect to $k^2$, we find that the large $n$ expansion of the ratio $D_n(k^2)/D_n(0)$ is given by
\bea\label{eq:R:Dn(k2)/D_n(0)_Large_n_Proof}
\frac{D_n(k^2)}{D_n(0)}&=&
\frac{\Big(1+\sqrt{1-k^2}\Big)^{2\bt(n+\al+\bt)}}{2^{2\bt(n+\al+\bt)}\Big(1-k^2\Big)^{\bt(\al+\bt/2)}}
\exp
\Bigg[
\frac{(4\al^2-1)\bt}{4n}\bigg(1-\frac{1}{\sqrt{1-k^2}}\bigg)
\nonumber\\&&
\qquad
-\frac{(4\al^2-1)\bt}{8n^2}\bigg(2\al+\bt-\frac{2(\al+\bt)}{\sqrt{1-k^2}}+\frac{\bt}{1-k^2}\bigg)+{\rm O}\left(\frac{1}{n^3}\right)
\Bigg],\qquad
\eea
which is exactly equation (\ref{eq:R:Dn(k2)/D_n(0)_Large_n}). 

To find the large $n$ expansion of $D_n(k^2)$, we multiply (\ref{eq:R:Dn(k2)/D_n(0)_Large_n_Proof}) by the large $n$ expansion of $D_n(0)$, given by (\ref{eq:R:D_n(0)_Large_n}), which then leads to (\ref{eq:R:Dn(k2)_Large_n}). Thus we have completed the proof to Theorem~\ref{thm:Large_n_Exp_Dn}. 
\end{proof}
\begin{remark}
At $\bt=0$, our weight (\ref{defn:R:w(xk2)}) reduces to the special case of the Jacobi weight $w^{(\al,\al)}(x)$. We can check by
substituting $\bt=0$ into (\ref{eq:R:Dn(k2)_Large_n}) that we obtain the correct large $n$ expansion for $D_n[w^{(\al,\al)}(\cdot)]$. 
\end{remark}

\section{Painlev\'e VI representation for Hankel determinant}
In this section we show by a change of variable how we can relate polynomials orthogonal with respect to (\ref{defn:R:w(xk2)}) to a set of polynomials orthogonal with respect to the following deformed shifted-Jacobi weight
\bea\label{eq:R:w2}
w_2(x,k^2,a,b,c)&=&x^{a}(1-x)^{b}(1-k^2x)^{c},\qquad x\in[0,1],\;\;a>-1,\;\;b>-1,\;\;c\in{\mathbb{R}},\;\; k^2\in(0,1).
\nonumber\\
\eea
The above weight was first studied by Magnus \cite{Magnus1995} (referred to as a generalized Jacobi weight with three factors), and more recently by Dai and Zhang \cite{DaiZhang2010}. Utilising the connection between the two weight functions, we are able to derive a correspondence between the three-term recurrence coefficients of the two sets of orthogonal polynomials. More importantly, we are able to derive a representation for our Hankel determinant (\ref{def:R:Hankel}) in terms of a Painlev\'e~VI (Theorem~\ref{thm:R:PVI_Rep}). %
\subsection{Relation to an alternate system of orthogonal polynomials}

Since the weight (\ref{defn:R:w(xk2)}) is an even function in $x$, it is possible to write (see \cite[p. 41]{NikiforovUvarov1988}) every even and odd normalization constant $h_{2n}$ and $h_{2n+1}$ as
\bea
\delta_{n,m}\;h_{2n}&=&\int\limits^1_0 Q_n(x)\;Q_m(x)\;x^{-1/2}\;(1-x)^\al{\;}(1-k^2x)^\bt {\,}dx,
\eea
and
\bea
\delta_{n,m}{\;}h_{2n+1}&=&\int\limits^1_0R_n(x){\;}R_m(x){\;}x^{1/2}{\;}(1-x)^\al{\;}(1-k^2x)^\bt {\,}dx,
\eea
respectively. In the above, we can treat $Q_n(x)$ and $R_n(x)$ as monic polynomials orthogonal with respect to the deformed shifted Jacobi weight (\ref{eq:R:w2}).

For even normalization constants $h_{2n}$, we consider monic polynomials $Q_n(x)$ orthogonal with respect to (\ref{eq:R:w2}), where $$a=-1/2,\qquad b=\al,\qquad c=\bt.$$
We denote the normalization constant, recurrence coefficients and sub-leading term of $Q_n(x)$ by $\widehat{h}_n$, $\widehat\al_n$, $\widehat\bt_n$  and $\widehat{\P}_1(n)$ respectively.

For odd normalization constants $h_{2n+1}$, we consider monic polynomials $R_n(x)$ orthogonal with respect to (\ref{eq:R:w2}), where
$$a=1/2,\qquad b=\al,\qquad c=\bt.$$
We denote the normalization constant, recurrence coefficients and sub-leading term of $R_n(x)$ by $\bar{h}_n$, $\bar\al_n$, $\bar\bt_n$  and $\bar{\P}_1(n)$ respectively.

The following relations then hold \cite{Chihara1978_OP}:
\begin{align}
\label{eq:AltPolyRelMinus}
h_{2n}&=\widehat{h}_n,& \bt_{2n+1}+\bt_{2n}&=\widehat\al_n,& \bt_{2n}\bt_{2n-1}&=\widehat\bt_n,& \P_1(2n)&=\widehat{\P}_1(n),
\\
\label{eq:AltPolyRelPlus}
h_{2n+1}&=\bar{h}_n,& \bt_{2n+2}+\bt_{2n+1}&=\bar\al_n,& \bt_{2n+1}\bt_{2n}&=\bar\bt_n,& \P_1(2n+1)&=\bar{\P}_1(n).
\end{align}

Through the above relations, we can calculate the asymptotic expansions of $\widehat{h}_n$, $\widehat\al_n$, $\widehat\bt_n$, $\widehat{\P}_1(n)$, $\bar{h}_n$, $\bar\al_n$, $\bar\bt_n$ and $\bar{\P}_1(n)$  for large $n$, since we have already calculated the large $n$ expansions of $\bt_n$, $\P_1(n)$, $D_n[w(\cdot,k^2)]$ and $h_n=D_{n+1}/D_n$.

\subsection{Proof of Theorem~\ref{thm:R:PVI_Rep}}
In this section, we characterize the Hankel determinant $D_n[w(\cdot,k^2)]$ for two cases: one where the matrix dimension $n$ is even, the other where it is odd.

Through the definition of the Hankel determinant $D_n$ in terms of $h_n$, (\ref{def:prodhn}), we can link the Hankel determinants generated by $w(x,k^2)$ and $w_2(x,k^2,a,b,c)$ by
\bea
\label{eq:R:D2n_Dn(w2)}
D_{2n}[w(\cdot,k^2)]=\prod\limits_{i=0}^{2n-1}h_{i}
&=&\left(\prod\limits_{i=0}^{n-1}h_{2i}\right)\left(\prod\limits_{j=0}^{n-1}h_{2j+1}\right),
\nonumber\\&&
=D_n[w_2(\cdot,k^2,-1/2,\al,\bt)]{\;}D_n[w_2(\cdot,k^2,1/2,\al,\bt)],\\
\label{eq:R:D2np1_Dn(w2)}
D_{2n+1}[w(\cdot,k^2)]=\prod\limits_{i=0}^{2n}h_{i}
&=&\left(\prod\limits_{i=0}^{n}h_{2i}\right)\left(\prod\limits_{j=0}^{n-1}h_{2j+1}\right),
\nonumber\\&&
=D_{n+1}[w_2(\cdot,k^2,-1/2,\al,\bt)]{\;}D_n[w_2(\cdot,k^2,1/2,\al,\bt)].
\qquad
\eea
We can use the ladder operator approach to characterize the Hankel determinant
generated by the weight $w_2(x,k^2,a,b,c)$. We define the function $\sig(k^2,n,a,b,c)$ through the Hankel determinant as
\bea\label{eq:R:sig(Dnw2)}
\sig(k^2,n,a,b,c)&=&k^2(k^2-1)\frac{d}{dk^2}\log D_n[w_2(\cdot,k^2,a,b,c)]+d_1k^2+d_0,
\eea
where
\bea
d_1&=&-nc-\frac{1}{4}(a+c)^2,\\
d_0&=&-\frac{n}{2}(n+a+b)+\frac{c}{4}(2n+a+b+c)-\frac{ab}{4}.
\eea
Based on the ladder operator approach used in \cite{DaiZhang2010}, $\sig(k^2,n,a,b,c)$
then satisfies the following Jimbo-Miwa-Okamoto $\sig$-form of Painlev\'e VI: \cite{JimboMiwa1981vII}
\bea
\sigma^\prime\Big(k^2(k^2-1)\sigma^{\prime\prime}\Big)^2+\Big\{2\sigma^\prime\big(k^2\sigma^\prime-\sigma\big)-
\big(\sigma^\prime\big)^2-\nu_1\nu_2\nu_3\nu_4\Big\}^2=\prod\limits_{i=1}^{4}\big(\nu_i^2+\sigma^\prime\big),
\eea
where ${}^\prime$ denotes differentiation with respect to $k^2$, and
\bea
\nu_1=\frac{1}{2}(c-a),\qquad\nu_2=\frac{1}{2}(c+a),\qquad\nu_3=\frac{1}{2}(2n+a+c),\qquad\nu_4=\frac{1}{2}(2n+a+2b+c).
\eea
In the above, due to the symmetry of the $\sig$-form, the parameters $\nu_1$--$\nu_4$ are not unique.

Hence the logarithmic derivative of the even Hankel determinants generated by $w(x,k^2)$, (\ref{defn:R:w(xk2)}),
\bea
H_{2n}(k^2)&:=&k^2(k^2-1)\frac{d}{dk^2}\log D_{2n}[w(\cdot,k^2)],
\eea
can be written using (\ref{eq:R:D2n_Dn(w2)}) and (\ref{eq:R:sig(Dnw2)}) as the following sum:
\bea
H_{2n}(k^2)&=&
k^2(k^2-1)
\bigg[\frac{d}{dk^2}\log D_n[w_2(\cdot,k^2,-1/2,\al,\bt)]
+\frac{d}{dk^2}\log D_n[w_2(\cdot,k^2,1/2,\al,\bt)]
\bigg],\\
&=&
\sig(k^2,n,-1/2,\al,\bt)
+\sig(k^2,n,1/2,\al,\bt)
+\Big(\frac{\bt^2}{2}+2n\bt+\frac{1}{8}\Big)k^2
-\frac{\bt}{2}(2n+\al+\bt)+n(n+\al),
\nonumber\\
\eea
where $\sig(k^2,n,-1/2,\al,\bt)$ and $\sig(k^2,n,1/2,\al,\bt)$
have representations in terms of the Painlev\'e VI $\sigma$-form. 
Similarly 
\bea
H_{2n+1}(k^2)&:=&k^2(k^2-1)\frac{d}{dk^2}\log D_{2n+1}[w(\cdot,k^2)],
\eea
can be written using (\ref{eq:R:D2np1_Dn(w2)}) and (\ref{eq:R:sig(Dnw2)}) as (\ref{eq:R:H2np1_PVI_Rep}), completing the proof of Theorem~\ref{thm:R:PVI_Rep}.
\section*{Acknowledgements}
Y. Chen would like to thank the Macau Science and Technology Development Fund for generous support (FDCT 077/2012/A3). N. S. Haq is supported by an EPSRC grant.

\appendix
\section{The coefficients of Theorem~\ref{thm:R:2nd_Ord_D_Eqn}}\label{App:2ndOrd_DE_Coeffs}
\begin{scriptsize}
\bea
{c}_{{0,0,0}}&= &(k^2-1)^2n\left( n+2\al \right)  \left(n+2\bt\right)
\left(n+2\,\alpha+2\,\beta \right). \\
{c}_{{0,1,0}}&=&
{\alpha}^2 (-3-4 {\beta}^2 +4 {\alpha}^2 ) ({k}^2 +1) ({k}^2 -1) ^2 (2 n+2 \alpha+2 \beta-3) (2 n+2 \alpha+2 \beta+3) (n+\alpha+\beta) ^2
\nonumber\\&&
-\frac{2}{9} (4\alpha^2-1)(\alpha^2-\beta^2) ({k}^2 +1) ({k}^2 -1) ^2 (2 n+2 \alpha+2 \beta-3) (2 n+2 \alpha+2 \beta+3) (2 n+2 \alpha+2 \beta-1) (2 n+2 \alpha+2 \beta+1)
\nonumber\\&&
-\frac{1}{9} (4\al^4-4 {\alpha}^2 {\beta}^2 -19\al^2-8 {\beta}^2  +18) ({k}^2 +1) ({k}^2 -1) ^2 (2 n+2 \alpha+2 \beta-1) (2 n+2 \alpha+2 \beta+1) (n+\alpha+\beta) ^2
\nonumber\\&&
  +2 ({k}^2 +1) (n+\alpha+\beta) ^2 -(\alpha^2-\beta^2)(16 n\alpha+16 \alpha \beta+1+16 n\beta+8 {n}^2 ) ({k}^2 -1) ^2
\nonumber\\&&+(4 \alpha \beta+2 {n}^2 -{\beta}^2 +5 {\alpha}^2 +4 n\alpha+4 n\beta) ({k}^2 -1) ({k}^2 +1).\\
 {c}_{{0,1,1}}&=&8k^2 \left( \alpha+\beta+n-\frac{3}{2} \right)
 \bigg[
 ({k}^2 -1) ^2 (\alpha+\beta+n) ^3 +\frac{1}{2} ({k}^4 +1) (\alpha+\beta+n) ^2
-(\al^2+{\beta}^2 ) ({k}^2 -1) ^2 (\alpha+\beta+n) +\frac{1}{2}({k}^4 -1) (\alpha^2-\beta^2)
 \bigg],\\
 {c}_{{1,1,0}}
&=&8k^2 \left( \alpha+\beta+n+\frac{3}{2} \right)
 \bigg[
 ({k}^2 -1) ^2 (\alpha+\beta+n) ^3 -\frac{1}{2} ({k}^4 +1) (\alpha+\beta+n) ^2
-({k}^2 -1)^2(\al^2+{\beta}^2 ) (\alpha+\beta+n) -\frac{1}{2}({k}^4 -1) (\alpha^2-\beta^2)
 \bigg],\\
 {c}_{{0,2,0}}&=&
-8 ({k}^4 +1) ({k}^2 -1) ^2 (n+\alpha+\beta) ^4 +24 ({k}^2 +1) ^2 ({k}^2 -1) ^2 (n+\alpha+\beta) ^4 -7 ({k}^4 +1) ({k}^2 +1) ^2 (n+\alpha+\beta) ^2 \nonumber\\&&
+({k}^4 +1) ({k}^2 -1) ^2 (n+\alpha+\beta) ^2 (8 {\al}^2+8\bt^2 +3 ) +16 ({k}^2 +1) (n+\alpha+\beta) ^2 ({k}^2 -1) ^2 ({k}^4 +1) (\alpha^2-\beta^2) \nonumber\\&&
-4 ({k}^2 +1) ^2 ({k}^2 -1) ^2 \Big(4 {k}^2 ({\alpha}^2-\bt^2) +6(\al^2+\bt^2) +1 \Big) (n+\alpha+\beta) ^2 +{\frac 9 8 } ({k}^4 +1) ({k}^2 +1) ^2
\nonumber\\&&
-6 ({k}^2 -1) ({k}^2 +1) ^2 ({k}^4 +1) (\alpha^2-\beta^2) +\frac{1}{8} ({k}^4 +1) ({k}^2 -1) ^2 (8 {\al}^2 +8 {\bt}^2 -1) ^2
-2 ({k}^2 +1) ({k}^2 -1) ^2 ({k}^4 +1) (\alpha^2-\beta^2)(4 {\alpha}^2+4 {\beta}^2+1 )
\nonumber\\&&
+\frac{1}{4} ({k}^4 -1) ^2 \Big(32 {k}^2({\alpha}^2-\bt^2) -8({\alpha}^2+\bt^2) -64 {\alpha}^2 {\beta}^2 +32 {k}^2 ({\alpha}^4-\bt^4) -1 \Big).\\
{c}_{{1,1,1}}&=&-8k^4\left( \alpha+\beta+n+\frac{3}{2} \right)  \left( \alpha+\beta+n-\frac{3}{2} \right)\bigg[(k^2+1)(\al+\bt+n)^2+(k^2-1)(\al^2-\bt^2)\bigg],\\
{c}_{{0,3,0}}&=&-64k^2
 \bigg[
 \frac{1}{8} ({k}^2 +1) (8{k}^4-15k^2+8) (n+\alpha+\beta) ^4
+{\frac 9 {32}} ({k}^2 +1) ({k}^4 +1) +{\frac 3 {64}} ({k}^2 +1) ({k}^2 -1) ^2 \Big(11 {k}^2 ({\alpha}^2 -{\beta}^2) -8 ({\alpha}^2+\bt^2) -2\Big)
\nonumber\\&&
-\frac{3}{64}(k^4+1)(k^2-1)(11k^2+8)(\al^2-\bt)^2 -{\frac {65}{64}} ({k}^4 +1) ({k}^2 +1) (n+\alpha+\beta) ^2
\nonumber\\&&
 -{\frac 1 {64}} ({k}^2 +1) ({k}^2 -1) ^2 (n+\alpha+\beta) ^2 \Big(36 {k}^2 ({\alpha}^2-\bt^2) +32 ({\alpha}^2+\bt^2) -1\Big)
\nonumber\\&&
+{\frac {17}{32}} (n+\alpha+\beta) ^2 ({k}^2 -1) ^2 ({k}^4 +1) (\alpha^2-\beta^2)  +\frac{1}{32}({k}^8 -1) (n+\alpha+\beta) ^2 (\alpha^2-\beta^2)
 \bigg],\\
 {c}_{{1,2,0}}&=&
 -32k^2\left( \alpha+\beta+n+\frac{3}{2} \right)
 \bigg[
 -({k}^2 +1) ({k}^2 -1) ^2\Big({\alpha}^2 +\frac{1}{4}\Big) (\alpha+\beta+n)  -\frac{1}{8} ({k}^4 -1) (\alpha+\beta+n) ({\alpha}^2 -{\beta}^2 +2)
 \nonumber\\&&
 +\frac{7}{8} ({k}^2 -1) ^2 (\alpha+\beta+n) (\alpha^2-\beta^2)  -\frac{1}{4} ({k}^2 +1) \Big(\alpha+\beta+n-\frac{3}{2}\Big) -\frac{1}{8}({k}^2 +1) ({k}^2 -1) ^2 (4\alpha^2-1)
\nonumber\\&&
 -\frac{1}{16}({k}^4 -1) (13 {\alpha}^2-13\bt^2-6 ) -{\frac{5}{16}}(k^2-1)^2 (\alpha^2-\beta^2)  -\frac{1}{8} ({k}^2 +1) (4{k}^4+5k^2+4)(\alpha+\beta+n)^2 \nonumber\\&&
 +\frac{1}{4}(k^2+1)(4k^4-9k^2+4)(\al+\bt+n)^3
 \bigg],\\
{c}_{{0,2,1}}&=&
 -32k^2\left( \alpha+\beta+n-\frac{3}{2} \right)
 \bigg[
 \frac{1}{4} ({k}^2 +1) (4{k}^4-9k^2 +4) (\alpha+\beta+n)^3
 +\frac{1}{8}({k}^2 +1) (4{k}^4+5k^2+4) (\alpha+\beta+n)^2
 \nonumber\\&&
  +\frac{1}{8} ({k}^4 -1) \Big({k}^2 (4{\alpha}^2-1) +5 {\alpha}^2 -9 {\beta}^2-2 \Big) -\frac{7}{4}(k^2-1) (\alpha^2-\beta^2)\left(\al+\bt+n+\frac{5}{14}\right) \nonumber\\&&
-\frac{1}{4}({k}^4 -1)  \Big({k}^2 (4{\alpha}^2+1) -7 {\alpha}^2 +3 {\beta}^2 \Big)(\alpha+\beta+n) -\frac{1}{4}({k}^2 +1) \left(\alpha+\beta+n+\frac{3}{2}\right)  \bigg].\\
{c}_{{1,2,1}}&=&32{k}^{4} \left( \alpha+\beta+n+\frac{3}{2} \right)  \left( \alpha+\beta+n-\frac{3}{2} \right)
\left[\left( \alpha+\beta+n+\frac{1}{2} \right)  \left( \al+\bt+n-\frac{1}{2}
 \right) \left(2k^4+k^2+2\right) -\frac{k^2}{2}
\right],
\eea\noindent
\bea
{c}_{{0,2,2}}&=&16k^4\left( \alpha+\beta+n-\frac{3}{2} \right)^{2} \Bigg[  \left( \beta+n+\frac{1}{2} \right)  \left(2\al+\bt+ n+\frac{1}{2} \right) {k}^{4}
+ \left( -3(\al+\bt+n)^2+\al^2+\bt^2+\frac{1}{2}\right) {k}^{2}
+ \left( \alpha+n+\frac{1}{2} \right)  \left( \alpha+2\bt+n+\frac{1}{2} \right)  \Bigg],\qquad\\
 {c}_{{2,2,0}}&=&
16k^4\left( \alpha+\beta+n+\frac{3}{2} \right)^{2} \Bigg[  \left( \beta+n-\frac{1}{2} \right)  \left(2\al+\bt+ n-\frac{1}{2} \right) {k}^{4}
+ \left( -3(\al+\bt+n)^2+\al^2+\bt^2+\frac{1}{2}\right) {k}^{2}
+ \left( \alpha+n-\frac{1}{2} \right)  \left( \alpha+2\bt+n-\frac{1}{2} \right)  \Bigg],\\
{c}_{{0,4,0}}&=&
k^4
\bigg[
32(3k^4-2k^2+3)(n+\al+\bt)^4
-8\Big((4\al^2+22)k^4+(21-4\al^2-4\bt^2)k^2+4\bt^2+22\Big)(n+\al+\bt)^2
\nonumber\\&&
 +\frac{9}{2} ({k}^4 +1) ({k}^2 -1) (8 {\al}^2 +8 {\bt}^2-1 ) +{\frac {99}2 } ({k}^2 +1) ({k}^4 +1)
-9 ({k}^4 -1) \Big(4 {k}^2\Big(\al^2 +{\beta}^2+\frac{5}{4}\Big)  +4(\al^2- {\beta}^2)  \Big)\bigg],
\nonumber\\
\\
{c}_{{1,3,0}}&=&32k^4\left( \alpha+\beta+n+\frac{3}{2} \right)
 \bigg[
 ({k}^4 -6k^2+1) (\alpha+\beta+n) ^3  -{\frac 9 8 } ({k}^4 +1) \left(\alpha+\beta+n-\frac{13}{6}\right) \nonumber\\&&
+\frac{1}{8}({k}^2 -1) (\alpha+\beta+n) (8 {k}^2 {\alpha}^2 -8\bt^2-{k}^2  +1) -\frac{1}{16} ({k}^2 -1) \Big(40 {k}^2 {\alpha}^2 -40\bt^2+17 ({k}^2-1)\Big) 
-\frac{1}{2} (7{k}^4+12k^2 +7) (\alpha+\beta+n) ^2
 \bigg],\\
{c}_{{0,3,1}}&=&
 32k^4 \left( \alpha+\beta+n-\frac{3}{2} \right)
\bigg[
-{\frac{9}{8} } ({k}^2 +1) \left(n+\alpha+\beta+\frac{13}{6}\right) -\frac{1}{2} ({k}^4 -1) (n+\alpha+\beta) \left({\beta}^2 -{\alpha}^2 +{\frac{9}{8} }\right)
\nonumber\\&&
+\frac{1}{16}({k}^2 -1) ^2 (n+\alpha+\beta) (8\al^2+8 {\beta}^2-11 ) + ({k}^4-6k^2 +1) (n+\alpha+\beta) ^3
\nonumber\\&&
+\frac{1}{16}(k^2-1)(40k^2\al^2-22k^2-40\bt^2-17)
 +\frac{1}{2} (7k^4+12{k}^2 +7) (n+\alpha+\beta) ^2
\bigg].\\
{c}_{{1,3,1}}&=&-128k^6 \left(k^2+ 1\right)  \left( \alpha+\beta+n+\frac{3}{2} \right)\left( \alpha+\beta+n-\frac{3}{2} \right)
 \left( (\al+\bt+n)^2-\frac{3}{4}\right),
\\ 
{c}_{{1,2,2}}&=&-2{k}^{6}(k^2+1) \left( 2\,\alpha+2\,\beta+2\,n-3 \right) ^{2} \left( 2\,\alpha+2\,\beta+2\,n+1 \right)
\left( 2\,\alpha+2\,\beta+2\,n+3 \right),\\
{c}_{{2,2,1}}&=&-2{k}^{6}(k^2+1) \left( 2\,\alpha+2\,\beta+2\,n-3 \right)  \left( 2\,\alpha+2\,\beta+2\,n+3 \right) ^{2}
 \left( 2\,\alpha+2\bt+2\,n-1\right),\\
{c}_{{0,4,1}}&=&
 32k^6(k^2+1)\left( \alpha+\beta+n-\frac{3}{2} \right)
 \left[\left( \alpha+\beta+n\right)^2\left( \alpha+\beta+n-\frac{19}{2} \right)
 +\frac{11}{4}\left( \alpha+\beta+n\right)+\frac{39}{8}
\right],\\
{c}_{{1,4,0}}&=&32k^6(k^2+1)\left( \alpha+\beta+n+\frac{3}{2} \right)
 \left[\left( \alpha+\beta+n\right)^2\left( \alpha+\beta+n+\frac{19}{2} \right)
 +\frac{11}{4}\left( \alpha+\beta+n\right)-\frac{39}{8}
\right],\\
{c}_{{0,3,2}}&=&2{k}^{6}(k^2+1) \left( 2\,\alpha+2\,\beta+2\,n-3 \right) ^{2} \left( 2\,\alpha+2\,\beta+2\,n+1 \right) \left( 6\alpha+6\bt+6n-7\right),\\
{c}_{{2,3,0}}&=&2\,{k}^{6}(k^2+1) \left( 2\,\alpha+2\,\beta+2\,n+3 \right) ^{2} \left( 6\,\alpha+6\,\beta+6\,n+7 \right)
 \left( 2\,\alpha+2\bt+2\,n-1 \right),\\
{c}_{{0,5,0}}&=&-64k^6(k^2+1) \left( \alpha+\beta+n+\frac{3}{2} \right)\left( \alpha+\beta+n-\frac{3}{2} \right)\bigg[\left( \alpha+\beta+n+\frac{1}{2} \right)\left( \alpha+\beta+n-\frac{1}{2} \right)-\frac{1}{2}\bigg].\\
{c}_{{2,4,0}}&=&-112k^8\left( \alpha+\beta+n+\frac{3}{2} \right)^2
 \left[\left( \alpha+\beta+n\right)^2
 +\frac{11}{7}\left( \alpha+\beta+n\right)-\frac{33}{28}
\right],\\
{c}_{{0,4,2}}&=&-112k^8\left( \alpha+\beta+n-\frac{3}{2} \right)^{2}
\left[ (n+\al+\bt)\left(n+\al+\bt-\frac{11}{7}\right)-\frac{33}{28}\right],\\ 
{c}_{{3,3,0}}&=&-4{k}^{8} \left( 2\,\alpha+2\,n-1+2\,\beta \right)  \left( 2\,\alpha+2\,\beta+2\,n+3 \right) ^{3},\\
{c}_{{0,3,3}}&=&-4\,{k}^{8} \left( 2\,\alpha+2\,\beta+2\,n-3 \right) ^{3} \left( 2\,\alpha+2\,\beta+2\,n+1 \right),\\
{c}_{{1,4,1}}&=&320k^8 \left( \alpha+\beta+n+\frac{3}{2} \right)\left( \alpha+\beta+n-\frac{3}{2} \right)\left( (\al+\bt+n)^2-\frac{13}{20}\right),\\
{c}_{{1,3,2}}&=&2\,{k}^{8} \left( 2\,\alpha+2\,\beta+2\,n-3 \right) ^{2} \left( 2\,\alpha+2\,\beta+2\,n+3 \right) ^{2},\\
{c}_{{2,3,1}}&=&2\,{k}^{8} \left( 2\,\alpha+2\,\beta+2\,n-3 \right) ^{2} \left( 2\,\alpha+2\,\beta+2\,n+3 \right) ^{2},\\
{c}_{{2,2,2}}&=&{k}^{8} \left( 2\,\alpha+2\,\beta+2\,n-3 \right) ^{2} \left( 2\,\alpha+2\,\beta+2\,n+3 \right) ^{2},\\
{c}_{{1,5,0}}&=&
-32k^8\left( \alpha+\beta+n+\frac{3}{2} \right) ^{2}
\left[ (n+\al+\bt)\left(n+\al+\bt+5\right)-\frac{15}{4}\right],\\
{c}_{{0,5,1}}&=&
 -32k^8\left( \alpha+\beta+n-\frac{3}{2} \right) ^{2}
\left[ (n+\al+\bt)\left(n+\al+\bt-5\right)-\frac{15}{4}\right],\\
{c}_{{0,6,0}}&=&{k}^{8} \left( 2\,\alpha+2\,\beta+2\,n-3 \right) ^{2} \left( 2\,\alpha+2\,\beta+2\,n+3 \right) ^{2}.
\eea
\end{scriptsize}


\bibliographystyle{model1b-num-names}
\bibliography{Reference_db}


\end{document}